\documentclass[11pt,a4paper,english]{article}

\usepackage{amssymb,amscd, amsmath, babel, amstext, amsthm, latexsym}

\newtheorem{thm}{Theorem}

\newtheorem{remark}{Remark}
\newtheorem{defn}{Definition}

\numberwithin{equation}{section}
\numberwithin{defn}{section}
\numberwithin{thm}{section}
\numberwithin{lem}{section}
\numberwithin{remark}{section}

\newcommand{\mcal}{\mathcal}

\newcommand{\R}{\mathbb{R}}

\newcommand{\F}{\ensuremath{\mathcal{F}}}
\newcommand{\U}{\ensuremath{\mathcal{U}}}
\newcommand{\V}{\ensuremath{\mathcal{V}}}

\newcommand{\A}{\ensuremath{\mathcal{A}}}
\newcommand{\B}{\ensuremath{\mathcal{B}}}
\newcommand{\T}{\ensuremath{\mathcal{T}}}

\newcommand{\ra}{\ensuremath{\longrightarrow}}

\newcommand{\st}{\ensuremath{s,t\in[0,T]:\,s\leq t}}
\newcommand{\ranget}{\ensuremath{t\in [0,T]}}

\newcommand{\pset}{\ensuremath{\mathbb P}}

\def \N{I\!\!N}

\def \esssup{\rm{esssup}}

\def \essinf{\rm{essinf}}

\def \AA{L_{\infty}({\cal A})}
\def \AAp{L^+_{\infty}({\cal A})}

\def \BB{L_{\infty}({\cal B})}
\def \BBp{L^+_{\infty}({\cal B})}

\def\ben{\begin{enumerate}}
\def\een{\end{enumerate}}
\def\bit{\begin{itemize}}
\def\eit{\end{itemize}}

\newtheorem{definition}{Definition}
\numberwithin{definition}{section}
\newtheorem{theorem}{Theorem}
\newtheorem{lemma}[theorem]{Lemma}
\newtheorem{corollary}[theorem]{Corollary}
\newtheorem{proposition}[theorem]{Proposition}

\numberwithin{theorem}{section}

\title{Extension theorems for linear operators on $L_\infty$ and application to price systems}
\author{Jocelyne Bion-Nadal\thanks{UMR 7641 CNRS - Ecole Polytechnique. Ecole Polytechnique, 91128 Palaiseau Cedex, France. Email: jocelyne.bion-nadal@cmap.polytechnique.fr}\: and Giulia Di Nunno\thanks{Centre of Mathematics for Applications (CMA), Department of Mathematics,
University of Oslo, P.O. Box 1053 Blindern, N-0316 Oslo Norway.
Email: giulian@math.uio.no}
\thanks{Norwegian School of Economics and Business Administration (NHH), Helleveien 30, N-5045 Bergen, Norway.}}
\date{February $25^{th}$, 2011}
\begin{document}
\maketitle

\vspace{-5mm}
\begin{abstract}
In an $L_\infty$-framework, we present a few extension theorems for linear operators. We focus the attention on \emph{majorant preserving} and \emph{sandwich preserving} types of extensions.
These results are then applied to the study of price systems derived by a reasonable restriction of the class of equivalent martingale measures applicable.
First we consider equivalent martingale measures with bounds on densities and the corresponding prices bounded by linear minorant and majorant. Then we consider prices bounded by bid-ask dynamics. Finally we study price systems consistent with no-good-deal pricing measures for given bounds on the Sharpe ratio. 
Within this study we introduce the definition of dynamic no-good-deal pricing measure.

\end{abstract}

%%%%%%%%%%%%%%%%%%
\section{Introduction}
The fundamental theorem of asset pricing is the key result celebrating the marriage between the economic principle of no-arbitrage and the mathematical tools of martingales and equivalent martingale measures. These provide the fundamental framework for pricing.
Several versions of this outstanding result have appeared with progressive improved level of generality, see e.g. \cite{DS2006}.
A crucial observation is that, provided existence, there is no uniqueness of equivalent martingale
measure guaranteed with the exception of markets that are complete, namely, in markets where all claims are attainable.
However, it is well-known that such markets are more a mathematical abstraction than proved reality and in general markets have to be considered incomplete.
As a consequence the problem of selecting one equivalent martingale measure out of the infinite many available has been largely treated.
The literature in this direction is vast and we chose not to mention any work in this direction.

More recently, starting with \cite{CSR2000} and \cite{BL2000} (see also \cite{C2003} and \cite{S2004}), a new interest developed.
Instead of selecting a single measure, one can restrict the set of equivalent martingale measures characterizing those that are in some sense ``reasonable".
The approach suggested is to rule out not only arbitrage opportunities, but also deals that are ``too good to be true".

In the same line, but with a different criterion, \cite{ADR2005} and \cite{DE2008} suggest to restric the set of equivalent martingale measures by choosing those with a density lying within pre-considered lower and upper bounds.
This criterion is motivated by the observation that some form of control on the so-called \emph{tail events}, i.e. crucial events appearing with small but positive probability, should be maintained when shifting from the physical measure (where statistical analysis is performed) to some equivalent martingale measure.

Another approach, developed in \cite{BN-dp}, consists in restricting the set of equivalent martingale measures to those compatible with bid and ask bounds observed for some traded options. This study is conducted under the more general setting of a time-consistent pricing procedure allowing for convex dynamic ask prices.

\vspace{2mm}
In the present paper we focus on linear price systems in incomplete markets that are consistent with pre-considered appropriate lower and upper bounds in connection with the various restrictions on the set of equivalent martingale measures. 
Thus we consider both the case where the density is lying within given bounds and the case where the bounds are set on the Sharpe ratio.
Moreover, we also deal with the characterization of linear price systems consistent with bid-ask bounds.

The approach we follow is independent of specific model for price dynamics.
We assume that prices $x_{st}(X)$, $0\leq s\leq t\leq T$, for marketed assets $X\in L_t$ are given and we describe them in axiomatic form. 
Here we set the bounds on prices $m_{st}(X) \leq x_{st}(X) \leq M_{st}(X)$ and we study the existence of a pricing measures $P_0$ that allows a linear representation 
\[
x_{st}(X) = E_{P_0}[X\vert \mcal{F}_s], \quad X\in L_t,
\]
fulfilling the given bounds.
The pricing measure $P_0$ will reflect the choices of bounds.

Note that the axiomatic presentation of a time-consistent price system, already present in \cite{BN} and \cite{DE2008}, is inspired by the literature in dynamic risk measures. 
See, e.g., \cite{P2004}, \cite{BEK}, and \cite{DPR}, in the context of Brownian filtrations; \cite{CDK} and \cite{KS2}, \cite{BN2008}, and \cite{BN2009}, for general filtered probability spaces, and \cite{BNK} in the case of model uncertainty.

The various applications are presented as result of a unique approach: the existence of a pricing measure allowing a linear representation of prices and fulfilling specific requirements corresponds to the possibility of extending the price operators onto the whole space in an appropriate way.
Thus we study extension theorems of linear operators.

This approach is already introduced in \cite{ADR2005}, and later developed in \cite{DE2008} to include the time-continuous case. 
However the present paper differs from these works in several ways. 
Papers  \cite{ADR2005} and  \cite{DE2008} study only applications to pricing measures with bounds on densities in the $L_p$-setting, while this contribution is framed an $L_\infty$-setting. 
We stress that a crucial difference is that $L_\infty$ spaces are not separable for the topology induced by the norm. Thus we cannot apply the same techniques as in \cite{DE2008}, but we have to rely on the theory of filters, see Appendix.
Moreover, we present a version of the sandwich extension theorem with substantially weaker assumptions. 
This turns out to be fundamental in the application to no-good-deal pricing systems.
Our study is based on a new point of view on the concept of Sharpe ratio bounds.
As a result we introduce the concept of a dynamic no-good-deal pricing measure in a model free setting. This definition generalizes the static notion of no-good-deal pricing measure to a continuous time framework.

This paper is organized as follows.
In Section 2 we present the basic definitions and the axiomatic description of price operators and price systems. 
Section 3, which is also of self-standing interest, is dedicated to the extension theorems for linear operators on $L_\infty$. Our result present conditions for the existence of extensions that are bounds preserving. 
We include the conditions for a topological version of the extension theorems.  
A first non-trivial application of the results of Section 3 is the version of the fundamental theorem of asset pricing as introduced in Section 4. This theorem characterizes the conditions for the existence of pricing measures consistent with pre-defined bounds on prices.
The theorem is presented for continuous-time trading models. 
Here the theory of filters is used.
Section 5 and 6 are dedicated to the application of this general result to the specific restrictions on prices and measures mentioned before.
First we consider the case of bounds on martingale measure densities, then the case of prices lying within the bid and ask dynamics. Finally we study the case of no-good-deal pricing measures.
This last part require the analysis and the eventual extension of the definition of bounds on the Sharpe ratio to their dynamic version.

%%%%%%%%%%%%
\section{Linear pricing rules}\label{linear pricing}

We consider a continuous time market model without friction on the time interval \([0,T],\;T>0.\) Let \((\Omega,\F,P)\) be a complete probability space equipped with the right-continuous filtration \(\mathbb F:=\{\F_t,\;0\leq t\leq T\}\)  with \(\F_T=\F\).

We work in an \(L_\infty\)-framework and consider claims as elements of the space \(L_\infty(\F_t):=L_\infty(\Omega,\F_t,P)\) with finite norm
\[
\|X\|_\infty:=\esssup \vert X\vert,\quad X\in L_\infty(\F_t).
\]
Whenever we use a superscript $+$ in the notation of a space, we refer to the corresponding cone of the non-negative elements.

For any time \(t\in[0,T],\) let
\begin{equation}
L_t\subseteq L_\infty(\F_t)
\end{equation}
denote the linear sub-space representing all {market} claims that are payable at time \(t.\)
Note that in a complete market $L_t=L_\infty(\F_t)$ for all $t\in[0,T]$. However, in general \(L_t\subsetneq L_\infty(\F_t)\) for some \(t\in [0,T]\).

A \emph{num\'eraire} $R_t$, $\ranget$, is fixed in the market.
This is an asset that is always payable, i.e. $R_t \in L_t$ for all $\ranget$, at the price $0<R_t<\infty$ $P$-a.s.
For simplicity in notation we will consider this to be $R_t \equiv 1$.
Then prices and discounted prices will coincide.
Having this in mind hereafter we will not distinguish between the two and simply refer to price operators.

\begin{defn} \label{def:po}
For any $\st$ , the operator \( x_{st}\) defined on \(L_t,\) with values in $L_\infty(\F_s)$ is a \emph{price operator} if it is
\begin{itemize}
\item
\emph{monotone}, i.e. for any $X',\,X''\,\in L_t$,
\begin{equation}
  x_{st}(X')\geq   x_{st}(X''),\quad X'\geq X'',\label{mono}
\end{equation}
\item
\emph{strictly monotone}, i.e. for any $X',\,X''\,\in L_t$, 
\begin{equation}
  x_{st}(X')>  x_{st}(X''),\quad X'> X'',\label{monostr}
\end{equation}
where the strict inequality sign is meant in the sense that in addition to the P-a.s. inequality ``$\geq$'', the strict inequality ``$>$'' is verified on a set of positive measure,
\item \emph{additive}, i.e. for any $X',\,X''\,\in L_t$, 
\begin{equation}
  x_{st}(X'+X'')=  x_{st}(X')+  x_{st}(X''), \quad X',\,X''\,\in L_t,
\label{addi}
\end{equation}
\item \(\F_s\)\emph{-homogeneous}, i.e.
\begin{equation}
  x_{st}(\lambda X)=\lambda   x_{st}(X)
\label{homo}
\end{equation}
for all \(X\in L_t\) and  \(\lambda\in L^+_\infty(\mcal{F}_s)\) such that \(\lambda X\in L_t\),
\item
and
\begin{equation}\label{normal}
x_{st}(1)=1.
\end{equation}
\end{itemize}
\end{defn}

Note that \eqref{normal} is justified by the choice of num\'eraire.
Note also that, from \eqref{addi}, we have that $  x_{st}(0)=0$.
As a consequence of \eqref{homo}-\eqref{normal}, $x_{tt}(X)=X$ for $X\in L_t$.
Moreover note that from \eqref{mono} and \eqref{normal} it appears natural that
\[
\Vert x_{st}(X)\Vert_\infty < \infty, \quad X\in L_\infty(\F_t).
\]
In fact, the following observation holds.
\begin{remark}
Any monotone linear operator $x: L_\infty(\B) \rightarrow L_\infty(\A)$ is continuous in the norm topology $\Vert \cdot\Vert_\infty$, for any $\sigma$-algebras $\A\subseteq \B$.
{\rm Indeed, this is easily seen as $- \Vert X\Vert_\infty 1 \leq X\leq \Vert X\Vert_\infty 1$, hence $\Vert x(X)\Vert_\infty \leq \Vert X\Vert_\infty \Vert x(1) \Vert_\infty$.
In this way the concept of \emph{tame} operator defined in \cite{DE2008} (see also \cite{ADR2005}) is directly embedded in the definition.}
\end{remark}

\begin{defn}
\label{d:below}
Let $s,t\in [0,T]: s\leq t$. 
The price operator $ x_{st}(X)$, is \emph{continuous from above P-a.s.} at $X \in L_t$ if for any non-increasing sequence $X_n \in L_t$ with limit $X \in L_t$ we have
\begin{equation}\label{leftcont}
x_{st}(X_n) \downarrow x_{st}(X), \quad n\to \infty \quad P-a.s.
\end{equation}
\end{defn}
Note that, for a monotone linear operator, continuity from above is equivalent to continuity from below.

\begin{defn}
The family of price operators $ x_{st}$, $0\leq s\leq t \leq T$ is \emph{right-continuous} at $s$ if, for every $X \in L_t$, 
\begin{equation}\label{rc}
x_{s't}(X) \rightarrow x_{st}(X), \quad s'\downarrow s \quad P-a.s.
\end{equation}
\end{defn}

\begin{defn} \label{pcons}
Let $\T \subseteq [0,T]$.
The family $x_{st}$, $s,t \in \T: s \leq t$, of operators $x_{st}(X)$, $X\in L_t$, is \emph{time-consistent} (in $\T$) if for all $s, u, t\in \T$: $s \leq u \leq t$
\begin{equation}\label{consist}
x_{st}(X)=x_{su}\big(x_{ut}(X)\big),
\end{equation}
for all $X\in L_t$ such that $x_{ut}(X) \in L_u$.
\end{defn}

In the sequel we will consider time-consistency \eqref{consist}. This is a natural assumption in view of standard arguments of absence of arbitrage.

\begin{defn}\label{ps}
A \emph{pricing system} is the whole time-consistent \eqref{consist}, right-continuous  \eqref{rc} family of price operators $x_{st}(X)$, $X\in L_t$, $0 \leq s \leq t \leq T$, continuous from above \eqref{leftcont}.
\end{defn}

%%%%%%%%%%%%%%%%%%%%%%%%%%%%%%%%%%
\section{Representation and extension theorems for operators on $L_\infty$}\label{sec:extension}

In this section we study extension theorems for operators which will be applied to the case of a single period market model with trading times $s,t$: $s\leq t$.
To keep the exposition general enough, we will consider simply two $\sigma$-algebras $\A\subseteq \B$.

\begin{definition}
\label{d:above}
A map $M: \BB \rightarrow  \AA$ is \emph{regular} if for every non-increasing sequence $X_n\in \BB$ with $X_n \downarrow 0$, $n\to \infty$ $P$-a.s, we have
\begin{equation}
\label{Mcont}
M(X_n) \rightarrow 0, \quad n \to \infty\quad P-a.s. 
\end{equation}
\end{definition}

%%%%%%
\subsection{Representation theorems and majorant conditions}
\label{Sec:Rep Thm}
Hereafter we deal with a representation theorem for  $\A$-homogeneous monotone linear operators continuous from above \eqref{leftcont} defined on  $\BB$ with values in $\AA$. 
The theorem relies on the following representation result for  positive linear forms on  $\BB$ continuous from above.
Even though the proof follows standard arguments, we could not find a reference for this result, hence we have chosen to present it fully.

\begin{lemma}
Let $L: \BB \rightarrow \R$ be a positive linear form continuous from above such that $L(1)=1$. Then there exists $f \in L^+_1(\B)$ , $E[f] =1$,  such that
\begin{equation}
L(X)=E\big[ fX \big],\qquad X \in \BB.
\label{eq1}
\end{equation}
\label{lemma1}
\end{lemma}
\begin{proof}
Denote ${\cal X}$ the space of bounded ${\cal B}$-measurable maps. Define  $\tilde L$ on ${\cal X}$ by $\tilde L (X)=L(\overline X)$, where $\overline X$ is the class of $X$ in $\BB$. From \cite[Appendix 50]{FS04}, there is a finitely additive set function $\mu$ on $(\Omega,{\cal B})$ with bounded total variation such that $L(X)=\int X d \mu$.  Let $B_n$ be an increasing sequence of events in $\Omega$ such that $\bigcup_n B_n=\Omega$. The sequence $1_{\Omega}-1_{B_n}$ is decreasing  to $0$. 
Hence $\mu(B_n)\uparrow \mu(\Omega)=1$, by application of the continuity from above and the additivity. 
Thus $\mu$ is a probability measure. Consider $B \in {\cal B}$ such that $P(B)=0$, i.e. $1_B=0$ in $\BB$. Then $\tilde L (1_B)=0$ and thus $\mu \ll P$. Hence there exists $f \in L^+_1(\B)$ such that equation (\ref{eq1}) is satisfied.
 \end{proof}

\begin{theorem}
Let $x: \BB \rightarrow \AA$ be an $\A$-homogeneous monotone linear operator continuous from above \eqref{leftcont}. Assume that there is a constant $c>0$ such that $x(1) \geq c$.
Then there is a probability measure $Q \ll P$ on $(\Omega,{\cal B})$ such that  
$$
x(X)=x(1) E_Q\big[X|{\cal A}\big] = x(1) E\Big[X \frac{f}{E[f|\mathcal{A}]}\vert \mathcal{A}\Big],
\qquad X \in \BB,
$$
and $f:= \frac {dQ}{dP} \in L^+_1(\B)$.
Moreover, there is a unique $f =\frac {dQ}{dP}$ in $L_1^+({\cal B})$ such that $E[f |\A]=1$ and $x(X)=x(1) E[fX|\A]$.
\label{thm1}
\end{theorem}
\begin{proof}
Since $x(1) \in \AA$ and $c\leq x(1)$. Then $x(1)^{-1} \in \AAp$. Denote $L(X)=E \big[ x(1)^{-1}x(X)\big]$.
From Lemma \ref{lemma1}, there is a probability measure $Q \ll P$ with $\frac {dQ}{dP} \in L^+_1(\B)$, such that $L(X)=E_Q [X]$. 
Let $A \in {\cal A}$. 
Applying the $\mathcal{A}$-homogeneity of $x$, we obtain:
\begin{eqnarray*}
E_Q[1_A x(X)]& =& E\big[ x(1)^{-1} x(1_A x(X)) \big]\nonumber \\
& =&  E\big[ x(1)^{-1}1_A x( X)x(1) \big] \nonumber\\
& =&  E\big[ x(1)^{-1} x(1_A Xx(1)) \big]\nonumber \\
& =&  E_Q[1_A Xx(1)].\nonumber 
\end{eqnarray*}
Hence, we have
$x(X)=E_Q[Xx(1)|{\cal A}]=x(1)E_Q[X|{\cal A}]$, $X \in \BB$.
\end{proof}

\vspace{2mm}
Recall that an operator $M:\BBp \rightarrow \AAp$ is \emph{sublinear} if
\begin{equation}
\label{sublinear}
M(X+Y) \leq M(X)+M(Y), \qquad  X,Y \in \BBp, 
\end{equation}
$$M(\lambda X)=\lambda M(X), \qquad X \in \BBp,\, \lambda \geq 0 .$$
We remark that sublinearity implies $M(0)=0$.

\vspace{2mm}
The following result shows that any $\A$-homogeneous monotone linear operator continuous from above admits some natural $\A$-homogeneous sublinear majorant.

\begin{corollary}
Let $x: \BB \rightarrow \AA$ be an $\A$-homogeneous monotone linear operator continuous from above \eqref{leftcont}. 
Assume that there is a constant $c>0$ such that $x(1) \geq c$. Then $x$ satisfies the majorant condition:
\begin{equation}
\label{eqmaj}
x(X) \leq M(Y), \qquad X \in \BB, Y \in \BBp:\; X \leq Y, 
\end{equation}
for some regular sublinear $\A$-homogeneous operator $M:\BBp \rightarrow \AAp$ \eqref{Mcont}.
\label{cor1}
\end{corollary}

\begin{proof}
From Theorem \ref{thm1}, there is a probability measure $Q \ll P$ associated to 
$x$ such that  $x(X)=x(1) E_Q[X|{\cal A}]$, $X \in \BB$.
Define  $M:\BBp \rightarrow \AAp$ by 
\begin{equation}
\label{M}
M(X):=x(X)=x(1)E_Q[X|{\cal A}], \qquad X \in \BBp.
\end{equation}
The operator $M$ is sublinear and it is  regular.
The majorant condition \eqref{eqmaj} is clearly satisfied.
\end{proof}

\vspace{2mm}
The following proposition shows that whenever a linear operator satisfies a majorant condition of type \eqref{eqmaj}, then it is also monotone and weak $\A$-homogeneous as defined here below.

\begin{definition}
We say that an operator $x:\BB \rightarrow \AA$ is \emph{weak $\mathcal{A}$-homogeneous} if, for every $X \in \BB$ and every $A \in {\cal A}$, it satisfies
\begin{equation}
\label{weak homogeneity}
x(1_A X)= 1_A x(X).
\end{equation}
\end{definition}

\begin{lemma}
\label{M1}
If $M: L^+_\infty(\mcal{B}) \rightarrow L^+_\infty(\mcal{A})$ is a weak $\mcal{A}$-homogeneous, regular, monotone, sublinear operator, then it is also $\mcal{A}$-homogeneous.
\end{lemma}
\begin{proof}
Recall that for $f\in  L^+_\infty(\mcal{A})$ there exists an increasing sequence of simple functions $f_n\uparrow f$, $n\to\infty$ in $ L^+_\infty(\mcal{A})$.
For any $X\in L^+_\infty(\mcal{B})$, sublinearity and monotonicity imply
\[
M(fX) \leq M((f-f_n)X) + M(f_nX) \leq M((f-f_n)X) + M(fX).
\]
Hence, from regularity and weak $\mcal{A}$-homogeneity, we conclude
\[
fM(X) = \lim_{n\to \infty}f_nM(X) = \lim_{n\to \infty}M(f_nX) = M(fX).
\]
with convergence in $\AA$.
\end{proof}

\begin{lemma}
\label{x1}
If $x: \BB \rightarrow \AA$ is a weak $\mcal{A}$-homogeneous monotone linear operator continuous from above, then it is $\mcal{A}$-homogeneous.
\end{lemma}
\begin{proof}
From the linearity and the weak $\A$-homogeneity of $x$ we obtain that for every non-negative simple ${\cal A}$-measurable real function $g$ we have
\begin{equation}
x(gX)=gx(X), \qquad X \in \BB. 
\label{eq_1}
\end{equation}
Any $f \in \AAp$ is the $P$-a.s. limit of an increasing sequence of non-negative simple
${\cal A}$-measurable real functions $f_n$. 
As $x$ is linear, monotone, and continuous from above (and hence from below), it follows that $x(fX)$ is the $P$-a.s. limit of the increasing sequence $x(f_nX)$. 
This, together with \eqref{eq_1}, yields $x(fX) = fx(X)$ for $X \geq 0$.
For a general $f\in \AA$ and $X \in \BB$ we have to apply that $f= f^+ -f^-$ and $X= X^+ - X^-$.
\end{proof}

\begin{proposition}
\label{prop1}
Let $x:\BB \rightarrow \AA$ be a  linear operator. 
Assume that the majorant condition \eqref{eqmaj}:
\begin{equation*}
x(X) \leq M(Y), \qquad X \in \BB, Y \in \BBp: \; X \leq Y,
\end{equation*}
is satisfied for some sublinear operator $M:\BBp \rightarrow \AAp$.
Then $x$ is monotone.
Moreover,
\begin{itemize}
\item[(i)]
if $M$ is weak $\A$-homogeneous, then $x$ is weak $\A$-homogeneous,
\item[(ii)]
if $M$ is regular, then $x$ is continuous from above.
\end{itemize}
\end{proposition}

\begin{proof}
Let $X \in \BBp$, then $-X \leq 0$. 
Thus the sublinearity of $M$ implies $x(-X)\leq M(0)=0$, i.e. $x(X) \geq 0$. The monotonicity of $x$ follows then from its linearity.

(i) \; Let $X \geq 0$. From the monotonicity of $X$, the majorant condition and the $\A$-homogeneity of $M$ we have
$$0 \leq x(1_AX) \leq 1_A M(X).$$
Thus $1_{A^c}x(1_AX)=0$ and $x(1_AX) = 1_A x(1_AX)$.
Similarly we prove that
$1_Ax(1_{A^c}X) = 0$.
Applying the linearity of $x$ we conclude that
$x(1_AX)=1_Ax(1_AX)+1_{A}x(1_{A^c}X) = 1_Ax(X)$.
Namely, weak $\A$-homogeneity \eqref{weak homogeneity} holds for $X \geq 0$.  For a general $X$ in $\BB$, write $X$ as $X= X^+ - X^-$,  where $X^+:=\max\{0,X\}$, $X^-:=\max\{0,-X\}$.
 By linearity, the weak homogeneity, equation (\ref{weak homogeneity}) 
extends to the whole $\BB$.

(ii) \; Let $ X_n \in \BB$ be an decreasing sequence with $P$-a.s. limit $X$. From the linearity, the monotonicity of $x$, and the majorant condition \eqref{eqmaj} we have
\begin{equation}
\vert x(X_n)-x(X) \vert \leq M(X_n-X).
\label{eq1.5}
\end{equation}
Since the sequence $X_n-X$ is decreasing to $0$, from the regularity of $M$, we obtain that $x$ is continuous from above.
\end{proof}

\begin{remark}
Note that if $M: L_\infty^+(\B) \rightarrow L_\infty^+(\A) $ is a sublinear operator such that
\[
x(X) \leq M(Y), \quad X\in L,\, Y \in L_\infty^+(\B):\: X\leq Y,
\]
for some linear operator $x: L\rightarrow\AA$, where $L \subseteq \BB$ is a linear subspace, then it is always possible to construct a monotone sublinear $\tilde M: L_\infty^+(\B) \rightarrow L_\infty^+(\A) $ such that
\[
x(X) \leq \tilde M(Y) \leq M(Y), \quad X\in L,\, Y \in L_\infty^+(\B):\: X\leq Y.
\]
Moreover, if $M$ is regular, then $\tilde M$ is regular.
\end{remark}
\begin{proof}
It is enough to consider $\tilde M(Y) := \inf_{Y'\geq Y} M(Y')$, $Y \in L_\infty^+(\B)$.
\end{proof}

%%%%%%%%
\subsection{A majorant preserving extension theorem}
\label{Sec:Extension}
 In the previous subsection we have  seen that  any $\A$-homogeneous monotone linear operator $x: \BB \rightarrow \AA$, continuous from above (such that $x(1)\geq c$ for some $c>0$)  satisfies the majorant condition \eqref{eqmaj}:
$$
x(X) \leq M(Y), \quad\ X \in \BB, \; Y \in \BBp:\; X \leq Y,
$$ 
for some regular sublinear $\A$-homogeneous operator $M:\BBp \rightarrow \AAp$ (see Corollary \ref{cor1}).
Now we prove that the majorant condition is a sufficient condition for a monotone linear operator defined on a linear subspace $L \subseteq \BB$ in order to have a linear monotone extension to the whole $\BB$.\\
In the sequel we assume that the $\sigma$-algebra ${\cal B}$ is generated by a countable family of events $A_n$, $n \in \N$, by which we mean that the $\sigma$-algebra ${\B}$ is the smallest $\sigma$- algebra on $\Omega$ containing both the sets $A_n, \;n \in \N$ and the $P$-null events. It is for example the case for the Borel $\sigma$-algebra of a metrizable separable space.

\begin{theorem}
Let $x$ be a monotone linear operator defined on a linear subspace $L$ of $\BB$. Assume that the majorant condition
\begin{equation}
x(X) \leq M(Y), \quad  X \in L,\;Y \in \BBp: \;X \leq Y,
\label{eqmaj3}
\end{equation}
is satisfied
for some regular, weak $\A$-homogeneous, and sublinear operator 
$M:\BBp \rightarrow \AAp$.
Then $x$ can be extended into a monotone linear operator defined on $\BB$ such that the majorant condition
\begin{equation}
x(X) \leq M(Y),\quad X \in \BB,\;Y \in \BBp: \;X \leq Y,
\label{eqmaj4}
\end{equation}
is satisfied. \\
Furthermore $x$ is continuous from above \eqref{leftcont} and $\A$-homogeneous.
\label{thm2.1}
\end{theorem}

\begin{proof}
As in the proof of Theorem 4.1 in \cite{ADR2005}, we begin by a one-step extension.
This is a classical approach already present in the original proof of the Hahn-Banach theorem.
Let $Y^0 \in \BB-L$. We want to extend $x$ by the formula $x(X+\lambda Y^0)=x(X)+\lambda Z$ for some $Z \in \AA$.
Let 
$$
a =\esssup_{X' \in L,Y'\in \BBp:\, -X'-Y'\leq Y^0} [-x(X')-M(Y')]
$$
$$
b=\essinf_{X'' \in L,Y''\in \BBp: \, X''+Y''\geq Y^0} [x(X'')+M(Y'')].
$$
Note that $ -X'-Y'\leq Y^0 \leq  X''+Y''$. Thus $-X'-X''\leq Y'+Y''$.
From the majorant condition (\ref{eqmaj3}) and the sublinearity of $M$ it follows that $-x(X')-x(X'')\leq M(Y')+M(Y'')$.
Thus $a \leq b$.
Choose $Z$ such that $a \leq Z \leq b$. It is then easy to verify that the extension of $x$ to the linear space $L+\R Y^0$ satisfies the majorant condition. 
Then, since $M(0)=0$, the monotonicity of the extension of $x$ follows from the majorant condition.\\
Now, let $A_n$, $n \in \N$, be a countable family of events in $\Omega$ generating the $\sigma$-algebra ${\B}$. 
Consider the linear subspace $K$ of $\BB$ generated by $L$ and the indicator functions $1_B$ where $B$ is the intersection of only a finite number of sets among $A_n$, $n \in \N$ and their complements $A_n^c=\Omega-A_n$, $n \in \N$. 
Applying the argument above, $x$ can be extended to $K$ as a linear monotone operator satisfying the majorant condition. \\
Let $E$ be a linear subspace of $\BB$. 
Assume that $x$ is extended to $E$ and that this extension is linear monotone and satisfies the majorant condition. 
Let $X_n$ and $Y_n$ be two  increasing sequences of elements of $E$ having the same limit $X \in \BB$. The sequences $x(X_n)$ and $x(sup(X_n, Y_n))$, are increasing and both majorized by $M (|X|)$. 
Therefore they converge in $\AA$ with limit $Y$ and $Z$, respectively, and such that $Y \leq Z$.
Note that
\begin{equation}
x(sup(X_n, Y_n))-x(X_n) \leq M(sup(X_n, Y_n)-X_n).
\label{eq3.1}
\end{equation}
The sequence $l_n :=\sup_{k \geq n}(sup(X_k, Y_k)-X_k)$ is decreasing and has limit $0$. 
Thus, as $M$ is regular, from \eqref{eq3.1} and the majorant condition we conclude that $Z-Y=0$. 
Then the sequences $x(X_n)$ and $x(Y_n)$ have the same limit.\\
In the same manner if  $X_n$ and $Y_n$ are  two  decreasing sequences of elements of $E$ having the same limit $X \in \BB$, the corresponding sequences $x(X_n)$ and $x(Y_n)$ have the same limit. 
Moreover, if $X_n$ is increasing to $X$ and $Y_n$ decreasing to $X$, from the majorant condition and the regularity of $M$, it follows that $x(Y_n)-x(X_n)$ has limit $0$.\\
Therefore $x$ can be extended in a unique way to a linear subspace $E$ of $\BB$ containing $K$ and containing the limit of all increasing and decreasing sequences of elements of $E$. From the monotone class theorem, it follows that $E$ contains $1_A$ for every set $A$ belonging to the $\sigma$-algebra generated by the sets $A_n $, $n=1,2,...$ (i.e. the $\sigma$-algebra $\B$). 
Recall that any non-negative ${\B}$-measurable function is the increasing $P$-a.s. limit of a sequence of linear combinations of indicators $1_A$, $A \in {\B}$. 
As $E$ is a sublinear space of $\BB$, this proves that $E=\BB$.
Furthermore this extension is obviously monotone. \\
Denote $S$ the subset of all $X \in  E$ satisfying the majorant condition, i.e. 
$x(X) \leq M(Y)$, $Y \in \BBp$: $X \leq Y$.
Then $S$ is obviously stable for the limit of increasing sequences. 
On the other hand, if $X \in E$ is the limit of a decreasing sequence $X_n$ of elements of $S$,  $X \leq Y$ and $X_n \leq sup(Y, X_n)$, then 
$x(X_n) \leq M(sup(Y, X_n)$. 
From the monotonicity and the sublinearity of $M$, it follows that $M(Y) \leq M(sup(Y, X_n)) \leq M(Y)+ M(sup(Y, X_n) -Y)$. 
As $M$ is regular, it follows that $M(sup(Y, X_n))$ has limit $M(Y)$. 
Thus we have
$$
x(X) = \lim_{n\to \infty} x(X_n) \leq \lim_{n\to \infty} M(\sup(Y,X_n)) = M(Y).
$$ 
Namely, $X$ satisfies the majorant condition.
Hence $X \in S$. 
We conclude that the set $S$ is stable for limits of both increasing and decreasing sequences. 
Hence $S=E=\BB$.
From Proposition \ref{prop1} we conclude directly that $x$ is continuous from above \eqref{leftcont}, from Lemma \ref{x1} we conclude that it is $\A$-homogeneous.
\end{proof}

%%%%%%%%%%%%%%%
\subsection{A sandwich preserving extension theorem}
\label{Sec:Sandwich}
This section deals with the  sandwich condition for operators in $L_{\infty}$-spaces and related extension theorems.

Let $M:\BBp \rightarrow \AAp$  be a  sublinear operator, see \eqref{sublinear}, and 
$m:\BBp \rightarrow \AAp$ be a \emph{superlinear} operator, i.e.
\begin{equation}
\label{superlinear}
m(X+Y) \geq m(X) + m(Y)
\end{equation}
$$
m(\lambda X) = \lambda m(X)\quad \lambda \geq 0.
$$
Let $x: \BB \rightarrow \AA$  be a  linear operator satisfying the \emph{sandwich condition}:
\begin{equation}
m(X) \leq x(X) \leq M(X), \quad \textrm{for all }X \in \BBp.
\label{mM}
\end{equation}
Note that the sandwich condition (\ref{mM}) is equivalent to the following condition:
\begin{equation}
\begin{split}
m(Z)+ &x(X'') \leq x(X')+M(Y), \\
& \textrm{for all }X', X'', Y,Z \in \BBp:\; Z+X'' \leq X'+Y.
\end{split}
\label{san0}
\end{equation}
Moreover, it is also equivalent to:
\begin{equation}
\begin{split}
m(Z)+& x(X) \leq M(Y), \\
&\textrm{for all } X \in \BB, \;  Y,Z \in \BBp:\; Z+X \leq Y.
\end{split}
\label{san01}
\end{equation}
To see that \eqref{san01} implies \eqref{san0}, it is enough to apply the first one with $X=X''-X'$.
Conversely, note that, for any $X \in \BB$, the elements $X^+$ and $X^-$ also belong to $\BB$.
We can then apply (\ref{san0}) with $X''=X^+$ and $X'=X^-$. 
Using the linearity of $x$, it is easy to see that \eqref{mM} is equivalent to \eqref{san01}.\\
Note that, in case our operator $x$ was defined on a convex cone instead of a linear subspace, then the sandwich condition should be expressed as \eqref{san0} only.

\vspace{2mm}
We adress the question of the existence of a sandwich extension to $\BB$ of a monotone linear operator $x$ defined on a linear subspace $L \subseteq \BB$ and taking values in $\AA$.
In \cite{ADR2005}, the characterization of the existence of such an extension was adressed for operators $x$ defined on convex subcones of $L_p(\B)$, $p\in [1,\infty)$. 
The proof given in \cite{ADR2005} is using the K\"onig sandwich theorem for functionals proved in \cite{FL} crucially relying on Zorn lemma.
In the present paper, we work in the context of $L_{\infty}$-spaces and we give a different constructive proof inspired by the proof of Theorem \ref{thm2.1}. 
This idea could also be applied in the $L_p$-context to give a new proof for Theorem 5.1 in \cite{ADR2005}, if the operators were defined on linear subspaces instead of convex subcones.

We stress that, for a general linear subspace $L \subseteq \BB$, the fact that $X \in L$ does not imply that $X^+ \in L$. 
The sandwich relation~(\ref{san01}), applied with $X\in L$, will then play a crucial role in the results that follow.

\begin{theorem}
Let $L$ be a linear subspace of $\BB$. 
Let $M:\BBp \rightarrow \AAp$  be a regular sublinear  operator and $m :\BBp \rightarrow \AAp$ be a superlinear operator. 
Let $x:L \rightarrow \AA$ be a linear operator satisfying the sandwich condition:
\begin{equation}
\begin{split}
m(Z)+& x(X) \leq M(Y), \\
&\textrm{for all } X \in L, \;  Y,Z \in \BBp:\; Z+X \leq Y.
\end{split}
\label{san}
\end{equation}
Then $x$ admits a monotone linear extension on the whole $\BB$. 
Moreover, the extension $x: \BB \rightarrow \AA$ is continuous from above \eqref{leftcont} and satisfies the
sandwich condition:
\begin{equation}
\begin{split}
m(Z)+& x(X) \leq M(Y), \\
&\textrm{for all } X \in \BB, \;  Y,Z \in \BBp:\; Z+X \leq Y.
\end{split}
\label{san2}
\end{equation}
which can equivalently be written as:
\begin{equation*}
m(X) \leq x(X) \leq M(X), \quad\ X \in \BBp.
\label{san3}
\end{equation*}
\label{thm3.1}
\end{theorem}
\begin{proof}
The proof follows the same lines as the proof of Theorem \ref{thm2.1}.\\
{\bf Step 1.} We begin with a one step extension.
Let $Y^0 \in \BB-L$. 
Let 
$$
c=\esssup_{X' \in L,Y', Z'\in \BBp:\, Z'-X'-Y'\leq Y^0}[m(Z')-x(X')-M(Y')]
$$
and
$$
d=\essinf_{X'' \in L,Y'',Z''\in \BBp:\, X''+Y''-Z''\geq Y^0}[x(X'')+M(Y'')-m(Z'')].
$$
Note that $ Z'-X'-Y'\leq Y^0 \leq  X''+Y''-Z''$. 
Thus $Z''+Z'-X''-X'\leq Y'+Y''$.
From the sandwich condition  (\ref{san}), the sublinearity of $M$, and the superlinearity of $m$, it follows that $m(Z'')+m(Z')-x(X'')-x(X')\leq M(Y')+M(Y'')$.
Hence $c \leq d$. 
Choose $y^0 \in \AA$ such that $c \leq y^0 \leq d$. 
Define the operator $x$ on $L +\R Y^0$ by $x(X+\lambda Y^0):=x(X)+\lambda y^0$.
We now prove that the sandwich inequality (\ref{san}) is satisfied for all $X\in L + \R Y^0$.
Let $ Y,Z \in \BBp$, $X \in L$, $\lambda \in \R$: $Z+X+ \lambda Y^0 \leq Y$. 
If $\lambda=0$, the sandwich condition is just the same as in the hypothesis.
Assume that $\lambda>0$.
Then $Y^0 \leq \frac{1}{\lambda}Y -\frac{1}{\lambda}Z-\frac{1}{\lambda}X$. 
Since $y^0 \leq d$, from the definition of $d$, it follows that $m(Z)+x(X)+\lambda y^0 \leq m(Z)+x(X)+ \lambda \Big(M(\frac{1}{\lambda}Y)-m(\frac{1}{\lambda}Z)+x(-\frac{1}{\lambda}X)\Big)=M(Y)$. 
Here we have applied the homogeneity of $m,\;M$ and the linearity of $x$.
This proves  the sandwich condition (\ref{san}) for $X\in L+\lambda Y^0$ with $\lambda >0$.
The proof of the sandwich inequality for $\lambda<0$ is similar, using the inequality $c \leq y^0$. 
Moreover, for every $X \in L  + \R Y^0$ and $Y \in \BBp$ such that $X \leq Y$, we have $x(X) \leq M(Y)$. 

{\bf Step 2.} As in the proof of Theorem \ref{thm2.1}, we proceed by extending $x$ to the sublinear space of $\BB$ generated by $L$ and the indicators $1_B$, where $B$ is either the intersection or the union of a finite number of $A_n$ or $A_n^c=\Omega -A_n$, $n \in \N$, generating $\B$.
Then we perform a further extension to the whole space $\BB$.
In order to prove that the sandwich condition  (\ref{san}) is satisfied, it is enough to show it for every $X \in \BB$ which is the increasing (and also for the decreasing) limit of a sequence $X_n$ of elements of $\BB$ satisfying the sandwich inequality. 
Let $Z,Y \in \BBp$ such that $Z+X \leq Y$. If $X$ is the increasing limit of $X_n$, by monotonicity of $x$, we just take the limit in the inequality
$m(Z)+x(X_n) \leq M(Y)$.
If $X$ is the decreasing limit of $X_n$, then $Z+X_n \leq Y+(X_n-X)$ and $m(Z)+x(X_n) \leq M(Y)+ M(X_n-X)$. Letting $n \to \infty$ and using the regularity of $M$, we complete the proof of the result.
\end{proof}

The preceding proof gives also a constructive new proof of Theorem 5.1. of \cite{ADR2005} in the case the operator $x$ is defined on a linear subspace $L\subseteq L_p(\mathcal{B})$, $1 \leq p < \infty$. 
Note that in the $L_p$-spaces, the continuity is with respect to the norm.

\begin{proposition}
 Let $L$ be a linear subspace of $L_p({\cal B})$, for $1 \leq p <\infty$. Let $M$  be a sublinear continuous operator 
$M:L_p({\cal B}) \rightarrow L_p({\cal A})$. Let $m$ be a superlinear operator
$m:L_{p}({\cal B}), \rightarrow L_{p}({\cal A})$. Let $x:L \rightarrow L_{p}({\cal A})$ be a linear operator.
Assume that the sandwich condition
\begin{equation}
\begin{split}
m(Z)+& x(X) \leq M(Y),\\
&\textrm{for all} Y,Z \in L_{p}({\cal B})\; X \in L:\; Z+X \leq Y,
\end{split}
\label{sand}
\end{equation}
is satisfied. 
Then $x$ admits  a monotone, continuous, linear extension preserving the
sandwich condition:
\begin{equation}
\begin{split}
m(Z)+& x(X) \leq M(Y), \\
&\textrm{for all } X \in L_p(\mcal{B}), \;  Y,Z \in L^+_p(\mcal{B}):\; Z+X \leq Y.
\end{split}
\label{sand2}
\end{equation}
This condition can also be written:
\begin{equation}
m(X) \leq x(X) \leq M(X), \quad \textrm{for all}  X \in L_{p}({\cal B}).
\label{sand3}
\end{equation}
\end{proposition}

\begin{proof}
{\bf Step 1.} 
The one-step extension is proved following the same lines as in Theorem \ref{thm3.1}.

{\bf Step 2.} 
Consider a countable family $f_n$ of elements of $L_{p}({\cal B})$ such that the linear subspace $K$ generated by $f_n$ is dense in $L_{p}({\cal B})$. 
Applying the first step we get the extension of $x$ to $L+K$, such that the sandwich condition is satisfied. 
From the majorant condition $x$ is continuous and thus uniquely extended to the whole $L_p(\mcal{B})$.
In order to prove the sandwich condition, for $Z+X \leq Y$ we consider $Y_n=Y+|X_n-X|$, then $Z+X_n \leq Y_n$. The sandwich inequality follows then from the norm continuity of $M$.
\end{proof}

%%%%%%%%%%
\subsection{Topological versions of the extension theorems}

Recall that the weak*topology on $\BB$ denoted $\sigma(  \BB, L^1({\cal B}))$ is the coarsest topology on $\BB$ such that for every $f \in  L^1({\cal B})$, the map $X \in \BB \;\rightarrow E(fX) \in \R$ is continuous.
\begin{lemma}Let $M: \BBp \rightarrow \AAp$ be a weak* continuous operator. Then $M$ is regular.
\label{lemmawcreg}
\end{lemma}
\begin{proof}
Let $X_n \in \BBp$ such that $X_n \downarrow 0$. From the dominated convergence theorem, $X_n \rightarrow 0$ for the weak* topoplogy, thus $M(X_n) \rightarrow 0$.
\end{proof}
 \begin{proposition}
Let $x: \BB \rightarrow \AA$ be a $\A$-homogeneous monotone linear operator. Assume that  $x(1)=1$. The following conditions are equivalent:
\begin{enumerate}
\item
$x$ is  continuous from below.
\item $x$ is  continuous from above.
\item $x$ is weak* continuous, which means that $x$ is continuous when both $ \BB$ and  $\AA$ are endowed with the weak* topology.
\end{enumerate}
\label{top1}
\end{proposition}

\begin{proof}
The equivalence of 1. and 2. follows from the linearity of $x$, considering  $-X, -X_n$ instead of  $X, X_n$.\\
Next we prove that {\it 2} implies {\it 3}. From Theorem \ref{thm1}, for every $x$ continuous from above there is a probability measure $Q$ such that
 $$x(X)= E_Q[X|{\cal A}], \quad X \in \BB.$$
 There is then $g \in L^+_1({\cal B})$ with $E[g|{\cal A}]=1$ such that $E_Q[X|{\cal A}]=E[gX|{\cal A}]$, $X \in \BB$. For every $f \in L_1({\cal A})$, $fg \in L_1({\cal B})$, indeed $E[|f|g]=E[|f|E[g|{\cal A}]]=E[|f|]$.
Assume now that $X_n \rightarrow X$ for the weak* topology. Let $f \in L_1({\cal A})$.
 $$E(fE_Q(X_n|{\cal A}))=E(fE(gX_n|{\cal A})=E(fgX_n)$$
 As $fg \in L_1({\cal B})$ and $X_n \rightarrow X$ for the weak* topology, it follows that $E(fgX_n) \rightarrow E(fgX)=E(fE_Q(X|{\cal A}))$. Thus $E_Q(X_n|{\cal A})\rightarrow E_Q(X|{\cal A})$ for the weak* topology of $\AA$. 
 This means that $x$ is weak* continuous. \\
Finally we prove that {\it 3} implies {\it 2}. Assume that $x$ is weak* continuous. Let $X_n \downarrow X$ i.e. $X_n -X \downarrow 0$. Thus from Lemma \ref{lemmawcreg}, and linearity of $x$ it follows that $x(X_n) \downarrow x(X)$. thus $x$ is continuous from above. 
\end{proof}
Now we can give a topological version of the Theorems \ref{thm2.1} and \ref{thm3.1}.
\begin{proposition}
Theorems \ref{thm2.1} and \ref{thm3.1} admit a topological version replacing in the hypotheses the regularity 
by the weak* continuity of $M$
and in
the conclusion the continuity from above of $x$ by its weak* continuity.
\label{rq3.3}
\end{proposition}
\begin{proof}
The result follows from Lemma \ref{lemmawcreg}, Proposition \ref{top1} and  from Theorems \ref{thm2.1} and \ref{thm3.1}.
\end{proof}

%%%%%%%%%%%%%%%%%%%%%%%%%%%
\section{A version of the fundamental theorem of asset pricing}\label{main}
In this section we consider a time-consistent family of price operators $x_{st},\st$, where $x_{st}: L_t \rightarrow L_s$. 
We assume that for every $t$, $L_t \subseteq L_T$.
\begin{remark}
For any $s\leq t \leq T$, $x_{st}$ is the restriction to $L_t$ of $x_{sT}$. 
\end{remark}
Indeed let $X \in L_t$, then $x_{tT}(X)=X x_{tT}(1)=X$. Thus  by time-consistency we have $x_{st}(X)=x_{st}(x_{tT}(X))=x_{sT}(X)$, for all $X\in L_t$.

We now introduce a definition of weak time-consistency for a family of sublinear (or superlinear) operators.
\begin{definition}
\begin{itemize}
\item 
The family $M_{st}$, $\st$, of $\mcal{F}_s$-homogeneous, sublinear operators $M_{st}: L_{\infty}^+({\cal F}_t) \rightarrow L_{\infty}^+({\cal F}_s)$ is \emph{weak time-consistent} if, for every $X \in L_{\infty}^+({\cal F}_t)$, 
\begin{equation}
M_{rs}(M_{st}(X)) \leq M_{rt}(X), \quad \forall r \leq s \leq t,
\label{wtcM}
\end{equation}
and
\begin{equation}
M_{st}(X)=lim_{t'>t,t'\downarrow t} M_{st'}(X).
\label{eqwtcM2}
\end{equation}
\item 
The family $m_{st}$, $\st$, of $\mcal{F}_s$-homogeneous, superlinear operators $m_{st}: L_{\infty}^+({\cal F}_t) \rightarrow L_{\infty}^+({\cal F}_s)$ is \emph{weak time-consistent} if, for every $X \in L_{\infty}^+({\cal F}_t)$, 
\begin{equation}
m_{rs}(m_{st}(X)) \geq m_{rt}(X),\quad\forall r \leq s \leq t,
\label{wtcm}
\end{equation}
and
\begin{equation}
m_{st}(X)=lim_{t'>t,t'\downarrow t} m_{st'}(X).
\label{wtcm2}
\end{equation}
\label{defwtc}
\end{itemize}
\end{definition}

\begin{remark} 
Every time-consistent family $M_{st}$ 
of $\mcal{F}_s$-homogeneous, sublinear operators such that (\ref{eqwtcM2})
is satisfied is weak time-consistent. 
Note that, if $M_{st}(1)=1$, then (\ref{eqwtcM2}) is trivially satisfied.
Similar arguments work for the superlinear case.
\end{remark}

\begin{thm}\label{thmP}
Let $M_{st}$, $\st$, be a weak time-consistent family of regular (or weak* continuous), $\mcal{F}_s$-homogeneous, sublinear operators $M_{st}: L^+_{\infty}({\cal F}_t )\rightarrow L^+_{\infty}({\cal F}_s)$;
let $m_{st}$, $\st$, be a weak time-consistent family of $\mcal{F}_s$-homogeneous, superlinear operators $m_{st}: L^+_{\infty}({\cal F}_t )\rightarrow L^+_{\infty}({\cal F}_s)$. 
Assume  that $m_{0T}(X)>0$ $P-a.s.$ for every $X> 0$ and that, for every $X \in L^+_{\infty}({\cal F}_t)$, for every sequence $s_n$ decreasing to $s$, we have
\begin{equation}
M_{st}(X) \geq \liminf M_{s_nt}(X);\;\;\;\;\;m_{st}(X) \leq \limsup m_{s_nt}(X)
\label{wrc}
\end{equation} 
Let 
\begin{equation}
\label{num}
x_{st}(X),\,X\in L_t,\,0\leq s\leq t\leq T,
\end{equation}
be a time-consistent and right-continuous family of price operators.
Suppose that the following  sandwich condition is satisfied:
\begin{equation}\label{sand1}
m_{st}(Z)+x_{st}(X) \leq M_{st}(Y)
\end{equation}
for all $X \in L_t$ and $Y,Z\in L^+_\infty(\F_t)$ such that $ Z+X\leq Y$.

Then there exists a probability measure $P^0\sim P$:
\begin{equation*}
P^0(A)=\int_A f(\omega)P(d\omega),\quad A\in\F,
\end{equation*}
with \(f\in L_1^+(\F)\) and $E[f|\F_0]=1$ such that
\begin{equation}
m_{st}(X)\leq E_{P_0}[X|\F_s] \leq M_{st}(X), \quad X \in L_{\infty}^+({\cal F}_t).
\label{bdens}
\end{equation}
and allowing the representation:
\[ 
x_{st}(X)=E_{P_0}[X|\F_s]= E\Big[X\frac{f}{E[f|\F_s]}|\F_s\Big],\quad X\in L_t,
\]
for all price operators. 
\end{thm}

Note that the last hypothesis on $M_{st}$ and $m_{st}$ (equation \ref{wrc}) is obviously satisfied if $M_{st}$ and $m_{st}$ are rightcontinuous in $s$.\\
The above theorem  appears in the same line as Theorem 4.1 in \cite{DE2008} where the study was carried out for operators in separable $L_p$-spaces with $1\leq p < \infty$, and for specific majorants and minorants. However we stress that the present result deals with weaker assumptions on the majorant and minorant operators. Moreover we remark  a crucial difference: the dual of $L_{\infty}$ endowed with the weak* topology is not metrizable. Then, to deal with the compactness features that follow, we call on the concept of \emph{filters}, see e.g. \cite{B90}. The most important notions used are summarized in the Appendix.

\begin{proof}
We have to prove that the set of probability measures
\begin{multline}\label{4addstar}
\pset:=\Big\{P^0\vert\;\frac{dP^0}{dP}=f, \; E[f\vert \F_0]=1 ,\;\forall s, t \in [0,T],\;\\\forall X \in L^+_{\infty}({{\cal F}_t}), m_{st}(X)\leq  E_{P_0}[X|\F_s] \leq M_{st}(X);\; \; \\ \forall X \in L_t, x_{st}(X)=E_{P_0}[X|\F_s]\; \Big \},
\end{multline}
is non-empty if \eqref{sand1} holds. 
We consider first the discrete time case
\begin{multline}\label{4addstar-discrete}
\pset^{(\T)}:=\Big\{P^0\vert\;\frac{dP^0}{dP}=f, \; E[f\vert \F_0]=1,\;\; \forall s,t \in \T,\;s \leq t\;,\\\forall X \in L^+_{\infty}({{\cal F}_t}), m_{st}(X)\leq  E_{P_0}[X|\F_s] \leq M_{st}(X);\; \; \\ \forall X \in L_t, x_{st}(X)=E_{P_0}[X|\F_s]\; \Big \},
\end{multline}
where \(\T\) is some partition of \([0,T]\) of the form
\begin{equation}
\T=\{s_0,s_1,\dots,s_K\}\text{, with  } 0=s_0< s_1<\dots<s_K= T.
\label{part}
\end{equation}
Further, we consider a sequence \(\{\T_n\}_{n=1}^\infty\) of increasingly refined partitions, such that \(\T_n\subset\T_{n+1}\) and mesh\((\T_n)\ra 0\) as \(n\ra\infty\). Clearly \(\pset^{(\T_{n+1})}\subset\pset^{(\T_n)}\). It is then sufficient to prove that
\renewcommand\theenumi{\Alph{enumi}}
\begin{enumerate}
\item
\(\pset^{(\T)}\) is non-empty for any finite partition \(\T\),
\label{discnone}
\item
the infinite intersection \(\bigcap_{n=1}^\infty\pset^{(\T_n)} \) is non-empty, and
\label{infintersec}
\item
any \(P^0\in \bigcap_{n=1}^\infty\pset^{(\T_n)}\) is also in \(\pset.\)
\label{intersec}
\end{enumerate}
To prove A, first of all note that by Theorem \ref{thm3.1} (or  Proposition \ref{rq3.3}), the sandwich condition \eqref{sand1} ensures that for every \(s\leq t\) the price operators \eqref{num} admit extensions $\tilde x_{st}$ on the whole $L_\infty(\F_t)$
and Theorem \ref{thm1} guarantees that there exists $f_{st}\in L_1^+(\mcal{F}_t)$: $E[ f_{st} \vert \F_s] =1$ such that
\begin{equation}\label{4add4}
\tilde x_{st}(X)=E\Big[Xf_{st}\big|\F_s\Big],\quad X\in L_\infty(\F_t).
\end{equation}

However, though the family of operators \eqref{num} is time-consistent, we cannot say, in general, that the extensions \eqref{4add4} are also time-consistent.
Then we proceed as follows.
Let us consider the partition points $\T$ and define
\begin{equation}\label{4add5}
f:=\prod_{k=1}^K f_{s_{k-1}s_k}.
\end{equation}
Define ${\hat x}_{st}(X):=E \Big[X\frac{f}{E[f|\F_{s}]}\big | \F_{s}\Big]$ $=E_{P_0}\big[X|\F_{s}\big]$, $X\in L_\infty(\F_t)$, where 
\begin{equation}
P_0(A)=\int_A f(\omega)P(d\omega),\quad A\in\F_T.
\label{eqQ0}
\end{equation}

 Then the family $\hat x_{st}$, $s,t \in [0,T]$  is time-consistent. Moreover for every  $X\in L_{s_k}$,  
\begin{equation*}
x_{s_{k-1}s_k}(X)=E\Big[X f_{s_{k-1}s_k}\big | \F_{s_{k-1}}\Big]=E\Big[X\frac{f}{E[f|\F_{s_{k-1}}]}\big | \F_{s_{k-1}}\Big]= \hat x_{s_{k-1}s_k}(X)
\end{equation*}
By iteration on $j-i$, it follows that for all $i \leq j$, 
 for every $X \in L_{s_j}$,
$$x_{s_{i}s_j}(X)=\hat x_{s_{i}s_j}(X)$$
  We can thus conclude that the probability measure $P_0$ defined by (\ref{eqQ0})
allows the representation
\[x_{st}(X)=x_{sT}(X)=E\Big[X\frac{f}{E[f|\F_s]}\big|\F_s\Big]=\hat x_{sT}(X) =\hat x_{st}(X) ,\quad X\in L_t,\]
for every $s \in\T$ and $t\in[s,T]$. Moreover, from Theorem \ref{thm3.1} or  Proposition \ref{rq3.3}, it follows from time consistency of $\hat x_{st}$ and weak time consistency of $m_{st}$ and $M_{st}$ that for every $s,t \in \T$,
$$m_{st}(X) \leq {\hat x}_{st}(X) \leq M_{st}(X), \quad  X \in L^+_{\infty}({\cal F}_t)$$ 

Thus \(\pset^\T\) is non-empty and \ref{discnone} holds.

\vspace{2mm}
The set \(\bigcap_{n=1}^\infty\pset^{(\T_n)}\) is non empty if the  corresponding sets  $ \pset^{(\T_n)}$ are \emph{weak\(^*\) compact}.
Here we are applying the \emph{finite intersection property}.

Then we have to prove that, for any partition (\T)  the set  \(\pset^{(\T)}\),   is weak\(^*\) compact.
As announced we use the concept of filters, see Appendix.

Denote $B^+$ the non negative part of the unit ball of the dual of $L_{\infty}({\cal F}_T)$. 
Note that 
\begin{multline*}
 \pset^{\T}:= \Big\{L \in B^+,\; L(1)=1, \forall s \leq t \in \T\; \forall A \in {\cal F}_s \\\forall X \in L^+_{\infty}({{\cal F}_t}), L(m_{st}(X)1_A)\leq  L(X1_A) \leq L(M_{st}(X)1_A);\; \; \\ \forall X \in L_t, L(x_{st}(X)1_A)=L(X1_A)\; \Big \},
\end{multline*}

Indeed the majoration $L \leq M_{0T}$ (which is a special case of the first inequality) implies from Proposition \ref{prop1} that $L$  is continuous from above i. e. that there is a probability measure $P_0$ such that $\forall X, \;L(X)=E_{P_0}[X]$. Furthermore, since $L$ belongs to the dual of $L_{\infty}({\cal F}_T)$, we conclude that $P_0 \ll P$. Then the second condition tells that $x_{st}(X)=E_{P_0}[X|{\cal F}_s]$ for all $X$ in $L_t$.\\
Note that
 $$\pset^{\T} = \; \pset^{\T}_1 \cap  \pset^{\T}_2$$
where

\begin{multline*}
\pset^{\T}_1 :=\Big\{L \in B^+, \forall s \leq t \in \T\; \forall X \in L^+_{\infty}({{\cal F}_t}), \forall A \in {\cal F}_s,\\L(m_{st}(X)1_A)\leq L(X1_A) \leq L(M_{st}(X)1_A)\;\Big \},
\end{multline*}
\begin{eqnarray*}
{\pset^{\T}_2}=\Big\{L \in B^+,\; L(1)=1,\;\forall s \leq t \in\T,\;\\
L(x_{st}(X)1_A)=L(X1_A) \;\; \forall A \in {\cal F}_s,\;\forall X\in L_t\Big\}
\end{eqnarray*}
We prove separately that both $\pset^{\T}_1 $ and ${\pset^{\T}_2}$  are weak* compact.

First we recall that $B^+$ is 
weak* compact. 

As the weak* topology is not metrizable, in order to prove that $\pset^{\T}_1 $ is a compact we show that every filter on $\pset^{\T}_1 $ has an adherent point.

Let $\U$ be a filter in $\pset^{\T}_1 $, then it is a base of filter in $B^+$.
As $B^+$ is compact, $\U$ has an adherent point here denoted $L$ in $B^+$.
Hence to prove compactness of $\pset^{\T}_1 $ we only need to verify that $L \in \pset^{\T}_1 $.

Denote $\V(L)$ the filter of the neighbourhoods $V(L)$ of $L$. In our context the neighbourhoods have the following form:
\[V(L)=
V_{\varepsilon, X_1,...,X_K}(L)
= \Big\{L'\in B^+:\; \vert L(X_k) - L'(X_k)\vert < \varepsilon, \:k=1,...,K\Big\}
\]
for $\varepsilon >0, K \in \mathbb{N}, X_1,...,X_K\in L_\infty(\F_T)$.
Recall that, being $L$ an adherent point, we have that $V(L) \cap U \ne \emptyset$ for every neighbourhood $V(L)$ and every $U\in \U$.
Let us consider $X\in L^+_\infty(F_t)$, $A\in \F_s$ and $\varepsilon >0$.
Let $L'\in V_{\varepsilon, X1_A,m_{st}(X)1_A,M_{st}(X)1_A}(L) \cap U\subseteq V_{\varepsilon, X1_A,m_{st}(X)1_A,M_{st}(X)1_A}(L)  \cap \pset^{\T}_1 $. Then, from
$$L'(m_{st}(X)1_A)\leq L'(X1_A) \leq L'(M_{st}(X)1_A)$$
and the inequalities
$$|L(m_{st}(X)1_A)-L'(m_{st}(X)1_A)| \leq \epsilon$$
$$|L(X1_A)- L'(X1_A)|\leq \epsilon$$
$$|L(M_{st}(X)1_A)-L'(M_{st}(X)1_A)| \leq \epsilon$$
letting $\varepsilon \rightarrow 0$, we conclude that $L \in \pset^{\T}_1 $.
 Hence $\pset^{\T}_1 $ is weak* compact.

In the case of ${\pset^{\T}_2}$, we consider $\U$ filter on ${\pset^{\T}_2}$. Since ${\pset^{\T}_2} \subseteq B^+$, we proceed with similar arguments.
To conclude that the adherent point $L \in B^+$ belongs to ${\pset^{\T}_2}$, we consider in particular the neighbourhoods $V(L)$ of  $L$ of type
\begin{multline*}
V(L)=V_{\varepsilon,1, 1_AX, 1_Ax_{st}(X)} (L) =
\Big\{L'\in B^+:\; |L'(1)-L(1)|\leq \epsilon,\vert \\L'(1_A  X) - L(1_A  X)\vert < \varepsilon,
 \vert L'(1_A  x_{st}(X)) - L(1_A x_{st}(X))\vert < \varepsilon\Big\}
\end{multline*}
for any $s \leq t \in \T$,  any $A\in \F_s$, any $X\in L_t\subseteq L_\infty(\F_T)$, and any $\varepsilon > 0$.
In this case an element $L' \in V_{\varepsilon,1, 1_AX, 1_Ax_{st}(X)} (L) \cap U \subseteq  V_{\varepsilon, 1,1_AX, 1_Ax_{st}(X)} (L) \cap {\pset^{\T}_2} $ satisfies
\[
L'(1_A  X) = L'(1_A  x_{st}(X)).
\]
Thus we have
\[
 \vert L(1_A  X) - L(1_A  x_{st}(X))\vert < 2 \varepsilon.
\]
\[
|L(1) -1| \leq 2 \varepsilon.
\]
Since the above estimate holds for every $A \in \F_s$, letting $\varepsilon \rightarrow 0$, we conclude that
the set ${\pset^{\T}_2}$ is  weak* compact.
This concludes the proof of B.

\vspace{2mm}
Assume that \(P_0\in \bigcap_{n=1}^\infty\pset^{(\T_n}).\)
As the partitions form a dense subset of \([0,T]\), then for any \(s\in [0,T]\) there is a sequence \(\{s_n\in\T_n\}_{n=1}^\infty\) such that \(s_n\downarrow s\) as \(n\ra\infty\). By the right-continuity of the filtration, and the right-continuity \eqref{rc} of the price operators we have
\[x_{st}(X)=\lim_{n\ra\infty} x_{s_nt}(X)=\lim_{n\ra\infty} E_{P_0}[X|\F_{s_n}]=E_{P_0}[X|\F_s]\quad X\in L_t\]
for any $s \in [0,T]$ and $t \in \bigcup \T_n$.
As $x_{st}(x)=x_{sT}(X)$ for every $X \in L_t$, the above equality is satisfied for every 
$s,t \in [0,T]$.
By the right-continuity of the filtration and the hypothesis (\ref{wrc}) on $M_{st}$ and $m_{st}$,   we also have  
for every  $s \in [0,T]$ and $t \in \bigcup \T_n$, 
$$m_{st}(X)\leq E_{P_0}[X|\F_s] \leq M_{st}(X)\;\;\forall X \in L_{\infty}(X)^+$$
Considering a sequence  \(\{t_n\in\T_n\}_{n=1}^\infty\) such that \(t_n\downarrow t\) as \(n\ra\infty\) as $M_{st}(X)=\lim M_{st_n}(X)$ and $m_{st}(X)=\lim m_{st_n}(X)$, see definition \ref{defwtc}, we conclude that  \(P_0\in\pset\) and \ref{intersec} holds.
\end{proof}

%%%%%%%%%%%%%%%%%%%%%%%%%%%
\section{Applications to price systems}\label{applications}
With this section we study several applications of the previous result. Our focus is in the characterization of price systems in connection with various forms of restrictions on prices or on the pricing measures.

\subsection{Pricing measures with bounds on density}\label{density}
First of all we present the case in which some restriction on the pricing measures is given in the form of lower and upper bounds for the martingale measure densities.
This criterion is motivated by the observation that some form of control on the so-called tail events should be maintained when shifting from the physical measure $P$, where statistical analysis is performed, to some pricing measure $P_0$.
This result is in line with \cite{ADR2005} and \cite{DE2008} where this application was studied for price operators in an $L_p$-setting. The first paper deals with the one-period market only, the second one extends this result to the dynamic framework. In \cite{DE2008}, some specific examples derived from insurance pricing can also be found.

\begin{proposition}
Let $m_{st},M_{st}\in L_1(\F_t)$, $0 \leq s \leq t \leq T$, such that
\begin{equation*}\label{4add1}
0<m_{st}\leq M_{st} \quad P-a.s.
\end{equation*}
and $m_{rs} m_{st}=m_{rt}$, $M_{rs}M_{st}=M_{rt}$, for any $r \leq s\leq t$. 
Assume that 
$m_{tt'}\rightarrow m_{tt}=1$ and $M_{tt'}\rightarrow M_{tt}=1$, for $t'\downarrow t$.
Then define 
\begin{equation}\label{5M}
\begin{split}
M_{st}(X)&:=E(M_{st}X|{\cal F}_s), \quad X\in L_\infty^+(\mcal{F}_t),\\
m_{st}(X)&:=E(m_{st}X|\F_s), \quad X\in L_\infty^+(\mcal{F}_t).
\end{split}
\end{equation}
Let 
$x_{st}(X)$, $X\in L_t$, $0\leq s\leq t\leq T$,
be price opertaors as in Theorem \eqref{thmP} satisfying the sandwich condition \eqref{sand1}.
Then there exists a probability measure $P^0\sim P$
allowing the representation:
\begin{equation*}
 x_{st}(X)=E_{P_0}[X|\F_s]= E\big[X\frac{f}{E[f|\F_s]}|\F_s\big],\quad \forall X\in L_t
\end{equation*}
with \(f\in L_1^+(\F)\) and $E[f|\F_0]=1$ such that
\begin{equation*}
m_{st}\leq\frac {E(f|{\cal F}_t)}{E(f|\F_s)}\leq M_{st}
\end{equation*}
\end{proposition}
\begin{proof} 
The operators $M_{st}$ and $m_{st}$ are linear, $\mcal{F}_s$-homogeneous, and regular. Moreover, the families $m_{st}$, $M_{st}$, $0\leq s\leq t\leq T$ are time-consistent and right-continuous. Then Theorem \ref{thmP} gives the result. 
\end{proof}

\begin{remark}
 A particular example of $m_{st}$,$M_{st} \in L_1({\cal F}_s)$, $0 \leq s \leq t \leq T$ satisfying the hypothesis of the theorem is given by  
\begin{equation}\label{4add2}
\begin{split}
M_{st}&:=\big(E[M|\F_0]\big)^{\frac{t-s}T}\frac{E[M|\F_t]}{E[M|\F_s]}\\
m_{st}&:=\big(E[m|\F_0]\big)^{\frac{t-s}T}\frac{E[m|\F_t]}{E[m|\F_s]}
\end{split}
\end{equation}
for $m,M>0$ $P-a.s.$ and $\frac{m}{E[m|\mcal{F}_s]}\leq \frac{M}{E[M|\mcal{F}_s]} \in L_1^+(\mcal{F}_T)$. In this case $m_{0T}=m\;\; M_{0T}=M$.
\end{remark}

%%%%%%%%%%%%%%%
\subsection{Price systems compatible with bid-ask dynamics}\label{bid-ask}
Delbaen has introduced in \cite{D} the notion of m-stability for a set of probability measures all absolutely continuous with respect to a given probability measure $P$. A set ${\cal Q}$ is m-stable if for all probability measures $Q_1 \ll P$ and $Q_2 \sim P$ in ${\cal Q}$ and for every  stopping time $\tau$ the probability measure  $Q$ such that $\frac{dQ}{dP}= (\frac{dQ_1}{dP})_{\tau}\frac{dQ_2}{dP}({\frac{dQ_2}{dP})^{-1}_{\tau}}$ belongs to ${\cal Q}$. We adopt the notation $(\frac{dQ}{dP})_{\tau}=E_P(\frac{dQ}{dP}|{\cal F}_\tau)$. ${\cal Q}$ also contains every probability measure whose Radon Nikodym derivative belongs to ${\cal F}_0$. It is proved in \cite{D} that every m-stable set ${\cal Q}$ defines a time consistent right continuous family of homogeneous superlinear operators : 
$M_{st}(X)=\sup_{Q \in \cal Q}E_Q(X|{\cal F}_s)$.\\
In the case of a non complete financial market, admitting no arbitrage, the set ${\cal M}$ of equivalent martingale measures for a family of reference assets is a m-stable set of equivalent probability measures (see \cite{D}). This motivates the following application:
\begin{proposition}
Let ${\cal Q}_1$ and ${\cal Q}_2$ be   m-stable subsets of the set ${\cal M}$ of equivalent martingale measures for the choosen reference assets. 
Let 
\begin{equation}
\label{baMm}
\begin{split}
M_{st}(X)&=\esssup_{Q \in{\cal Q}_1}E_Q[X |{\cal F}_s]\\
m_{st}(x)&=\essinf_{Q \in{\cal Q}_2}E_Q[X |{\cal F}_s]
\end{split}
\end{equation}
Assume  that $m_{0T}(X)>0$ $P-a.s.$, for every $X>0$.
 Every time-consistent and right-continuous family of price operators satisfying the sandwich condition can be extended into a  time-consistent and right-continuous family of price operator
$x_{st}$ such that $x_{st}(X)=E_{P_0}[X|{\cal F}_s]$, $X \in L_{\infty}({\cal F}_t)$, where $P_0$ is an equivalent martingale measure  satisfying 
\begin{equation}
\label{star}
\essinf _{Q \in{\cal Q}_2}E_Q[X |{\cal F}_s] \leq E_{P_0}[X|{\cal F}_s] \leq \esssup _{Q \in{\cal Q}_1}E_Q[X |{\cal F}_s], \quad X \in L_{\infty}({\cal F}_T).
\end{equation}
\label{corstable}
\end{proposition}
\begin{proof} As ${\cal Q}_1$ and ${\cal Q}_2$  are   m-stable sets of probability measures all equivalent with $P$, $M_{st}$ and $m_{st}$ satisfie all the hypothesis of Theorem \ref{thmP}.
\end{proof}

Note that from the linearity of $E_{P_0}[X§\F_s]$, multiplying the inequality \eqref{star} by $-1$, we obtain that 
\[\begin{split}
- \tilde M_{st}(-X) = &\essinf_{Q \in \mcal{Q}_1\cap{\cal Q}_2}E_Q[X |{\cal F}_s]\leq E_{P_0}[X|{\cal F}_s] \\
& \quad\leq \esssup _{Q \in\mcal{Q}_1\cap{\cal Q}_2}E_Q[X |{\cal F}_s]= 
\tilde M_{st}(X),
\end{split} 
\]
for all $ X \in L_{\infty}({\cal F}_T)$.
The set $\mcal{Q}_1\cap{\cal Q}_2$ is also m-stable
For every $t$, the process $\tilde M_{st}(X)$, $0\leq s\leq T$, admits then a c\`adl\`ag version for every $X$ (see \cite{D}).
Thus it has an extension to stopping times.
The operator $\tilde M_{\sigma \tau}$, $0\leq \sigma\leq \tau \leq T $ is then a no-free-lunch sublinear time-consistent pricing procedure according to the definition in \cite{BN}.

In a summary Proposition \ref{corstable} tells that every linear price system defined on the subspace $L_t$ of marketed assets at time $t$, and compatible with the bid and ask dynamics associated to a no-free-lunch time consistent pricing procedure can be represented by an equivalent martingale measure itself compatible with the bid and ask dynamics.

%%%%%%%%%%%%%%%%%%%%%%%%%%%%%%%%%%%%%%%%%%%%%%%%%%
\section{Dynamic no  good-deal price systems}
Good-deal bounds were introduced simultaneously by Cochrane and Saa Requejo \cite{CSR2000} and Bernardo and Ledoit \cite{BL2000} as a way to restrict the choice of equivalent martingale measures in incomplete markets. 
The idea is to consider martingale measures that not only rule out arbitrage possibilities, but also deals that are ``too good to be true''. 
Following \cite{CSR2000} we consider the characterization of no good-deals based on a restriction of the Sharpe ratio. 
Cochrane and Saa Requejo \cite{CSR2000}, and then Bj{\"o}rk and Slinko \cite{BS2006} start from a specific model for the traded assets: diffusions in \cite{CSR2000} and more general processes including jumps in \cite{BS2006}. They define an upper good deal price  process restricting the set of equivalent martingale measures.  Their definition of this set of measures strongly depends on the shape of the dynamics of the traded assets. Kl\"oppel and Schweizer \cite{KS} introduce a utility-based approach to restrict the set of equivalent probability measures. In the case of the exponential utility, and in the particular case of the completed  filtration generated by a Levy process, a m-stable set of equivalent martingale measures is constructed. Thus $M_{st}(X)=\esssup_{Q \in {\cal Q}}E_Q(X|{\cal F}_s)$ defines a time-consistent family of sublinear regular operators. However  the definition of ${\cal Q}$ relies on the particular shape of the densities of equivalent probability measures in the filtration of a Levy process, furthermore it is not closely related to the Sharp ratio.  For further discussion on the link between risk measures and no good deal pricing, we refer to \cite{JK} in the static case and to \cite{BN} in the dynamic case.\\
Hereafter we study the bounds on the Sharpe ratio to extract the minorant and majorant operators bounding the prices. We consider first the static one-period setting and then the multi-period one. We then motivate and extend in an appropriate way the definition of good-deal bounds to a continuous time dynamic version.
To this purpose we interpret the bounds on the Sharpe ratio as bounds on the Radon-Nykodim density of the corresponding equivalent martingale measure.
As usual in this paper we work with general price systems and not with specific price dynamics. In doing this we differ from large part of the literature related to good-deal bounds on continuous time market models.
However, to keep the presentation consistent, we work with payoffs in $L_{\infty}({\cal F}_T)$. This can be motivated recalling that if marketed assets are locally bounded adapted processes, then the corresponding stopped processes by a given stopping time are uniformly bounded.

\subsection{Static setting: one-period market}

Following Cochrane and Saa Requejo \cite{CSR2000}, a {\it good-deal of level $\delta>0$} is a non-negative ${\cal F}_T$-measurable payoff $X$ such that
\begin{equation*}
\frac{E(X)-E_Q(X)}{\sqrt {Var(X)}} \geq \delta.
\end{equation*}
Accordingly, a probability measure $Q$ equivalent to $P$ is a {\it no good-deal pricing measure} if there are no good-deals of level $\delta$ under $Q$, i.e., 
\begin{equation}
\label{eqSR}
E_Q[X]\geq E[X] - \delta \sqrt{Var(X)}, \quad X \geq 0.
\end{equation}
Note that \eqref{eqSR} holds for all $X\in L_\infty({\mcal{F}_T})$ as we have that $X+\Vert X \Vert_\infty \geq 0$.
Hence also the relation
\begin{equation}
\label{eqSR-b}
E_Q[X]\leq E[X] + \delta \sqrt{Var(X)}
\end{equation}
holds true for all $X\in L_\infty({\mcal{F}_T})$.
This motivates the following extended general definition of no-good-deal pricing measure.
\begin{definition}
\label{no-good-deal measure}
A probability measure $Q$ equivalent to $P$ is a {\it no good-deal pricing measure} if there are no good-deals of level $\delta>0$ under $Q$, i.e., 
\begin{equation}
\label{eqSR-2}
-\delta \leq \frac{E(X)-E_Q(X)}{\sqrt {Var(X)}} \leq \delta,
\end{equation}
for all $X\in L_2(\mcal{F}_T,P)\cap L_1(\mcal{F}_T,Q)$.
\end{definition}

In this forthcoming application $\mcal{F}_0$ is the trivial $\sigma$-algebra.
At first we study the static setting of a one-period market with trading times $0,T$ and we consider a linear pricing operator $x_{0T}$ defined on the linear subspace $L_T \subseteq L_{\infty}({\cal F}_T)$ representing the marketed assets. 
Hence we assume that 
$$
m(Z)+x_{0T}(X) \leq M(Y),
$$
for every $X \in L_T$ and $Y,Z, \in L^+_{\infty}({\cal F}_T)$: $Z+X \leq Y$.
Here we have considered $m(X):= E(X)- \delta \sqrt {Var(X)}$ and $M(X):=E(X)+ \delta \sqrt {Var(X)}$, $X\in L^+_{\infty}({\cal F}_T)$ for some positive $\delta$.

\begin{proposition}
The functionals $m(X) = E(X)- \delta \sqrt {Var(X)}$ and $M(X)=E(X)+ \delta \sqrt {Var(X)})$, $X\in L^+_\infty(\mcal{F}_T)$, are respectively superlinear and sublinear.
Moreover, the operator $M$ is regular. 
\label{lemmaE1}
\end{proposition}
\begin{proof}
The functionals $m$ and $ M$ are clearly homogeneous. 
By application of the Cauchy-Schwarz inequality, we see that, for every $X,Y$, $E(X+Y)^2 \leq   (\sqrt {E(X^2)}+ \sqrt {E(Y^2)})^2$. 
Then superlinearity of $m$ and sublinearity of $M$ follow.
The regularity of $M$ follows from the dominated convergence theorem.
\end{proof}

Being the conditions of Theorem \ref{thm3.1} satisfied, we can conclude that $x_{0T}$ admits an extension into a linear pricing operator defined on the whole $L_{\infty}({\cal F}_T)$ preserving the sandwich condition with price bounds given by the no good-deal restriction.

\subsection{Dynamic setting: multi-period market}

In the remain of this section we discuss extensions of the previous approach to a dynamic setting. 
First we consider a multi-period market model with trading times $0=s_0<s_1<\dots < s_k=T$.

\begin{lemma} 
\label{lm-dyn}
For any $s\leq t$, let
$c_{st}: L_{\infty}({\cal F}_t) \rightarrow L_{\infty}({\cal F}_s)$ be defined by
$c_{st}(X)=$ $ \delta_{st} \sqrt{\big(E[X^2|{\cal F}_s]\big)}$ for some positive $\delta_{st}$.
The operator $c_{st}$ is $\mcal{F}_s$-homogeneous, sublinear, and regular. 
Assume that, for every $r \leq s \leq t$, $\delta_{rt}=\delta_{rs}\delta_{st}$. 
Then $c_{st}$, $s,t\in [0,T]:$ $s<t$, is a time-consistent family of operators.
\end{lemma}

\begin{proof}
For any $s\leq t$, the operator $c_{st}$ is trivially $\mcal{F}_s$-homogeneous.
Sublinearity follows directly from the following inequality:
$$
E\big[(X+Y)^2|{\cal F}_s\big] \leq   \Big( \big( E[X^2|{\cal F}_s]\big)^{\frac {1}{2}}+
\big( E[Y^2|{\cal F}_s]\big)^{\frac {1}{2}}\Big)^2,
$$
that follows from the conditional H\"older inequality:
$$
E[XY|{\cal F}_s] \leq \big( E[X^2|{\cal F}_s]\big)^{\frac{1}{2}} \big(E[Y^2|{\cal F}_s]\big)^{\frac{1}{2}}.
$$
Consider $r \leq s \leq t$. Then we have 
$$
c_{rs}(c_{st}(X))
=\delta_{rs} \delta_{st} \Big( {E\big[ \big( E[X^2|{\cal F}_s] \big)^{\frac{1}{2}}}^2|{\cal F}_r\big]\Big)^{\frac{1}{2}}
= c_{rt}(X)
$$ 
This proves the time consistency.
The regularity follows directly from the dominated convergence theorem.
\end{proof}

The following proposition appears as a direct application of the sandwich extension theorem \ref{thm3.1}.

\begin{proposition}
Let ${\cal T}$ be a finite subset of $[0,T]$. Fix ${\cal T}:=\{s_0,s_1,.. s_k\}$. 
For every $0 \leq i \leq k-1$, let 
\[
\begin{split}
M_{s_i, s_{i+1}} &= E(X|{\cal F}_{s_i})+\delta_{s_i,s_{i+1}}\sqrt {(E((X-E(X|{\cal F}_{s_i}))^2|{\cal F}_{s_i})}, \quad X \in L^+_{\infty}({\cal F}_{s_{i+1}}),\\
m_{s_i, s_{i+1}} &=E(X|{\cal F}_{s_i})-\delta_{s_i,s_{i+1}} \sqrt {(E((X-E(X|{\cal F}_{s_i}))^2|{\cal F}_{s_i})}, \quad X \in L^+_{\infty}({\cal F}_{s_{i+1}}).
\end{split}
\]
Define recusively $ M_{s_i,s_j}$ for $j-i>1$ by $M_{s_i,s_j}(X)=M_{s_i,s_{i+1}}(M_{s_{i+1},s_j}(X))$, similarly for $m_{s_i,s_j}$.
Hence $M_{s_i, s_{j}}(X)$, $X \in L^+_{\infty}({\cal F}_{s_{j}})$, $0 \leq i \leq j \leq k$, is a time-consistent  family of regular $\mcal{F}_{s_i}$-homogeneous and sublinear operators with $M_{s_i, s_{j}}(1)=1$; $m_{s_i, s_{j}}(X)$, $X\in L^+_{\infty}({\cal F}_{s_{i+1}})$, $0 \leq i \leq j \leq k$, is a time-consistent family of superlinear operators with $m_{s_i, s_{j}}(1)=1$. \\
Let $x_{s_i,s_{j}}$, $i=0,...,k-1$, be a family of time-consistent price operators defined, for each $j$, on the linear subspace $L_{s_{j}}\subseteq L_{\infty}({\cal F}_{s_{j}})$ with values in  $L_{\infty}({\cal F}_{s_{i}})$ and satisfying the following sandwich inequality:
$$
m_{s_i, s_{i+1}}(Z)+x_{s_i, s_{i+1}}(X) \leq M_{s_i, s_{i+1}}(Y)
$$
for all $X \in L_{s_{i+1}}$ and $Y,Z, \in L^+_{\infty}({\cal F}_{s_{i+1}})$: $Z+X \leq Y$.\\
Then $x_{s_{i},s_{i+1}}$,$i=0,...,k-1$, extends into a time-consistent family of linear price operators preserving the sandwich condition above.
\label{propSR}
\end{proposition}

\vspace{2mm}
Note that in the result above, each operator $x_{s_i,s_{i+1}}$ satisfies the sandwich inequality with majorant and minorant directly connected with the Sharpe ration bounds, as in the static setting.
However, when we consider the operator $x_{s_i,s_j}$ ($i+1<j$), which also satisfies the sandwich inequality, the structure of the majorant and minorant operators is more complicated as it is defined by composition:
$M_{s_i,s_{j}}(X) = M_{s_i,s_{i+1}}(M_{s_{i+1},s_{i+2}}(...(M_{s_{j-1},s_{j}}(X))))$, and similarly for $m_{s_i,s_{j}}(X)$. 
Note, in fact, that the use of the time consistency of the corresponding family $c_{s_i,s_{i+1}}$, $i=0,...,k-1$, as given in Lemma \ref{lm-dyn} does not really help in finding simple expressions for $M_{0T}$ and $m_{0T}$. 
Hence, in this case, it is not easy to compare these values with the Sharpe ratio bounds and 
this theorem cannot be generalized to continuous time. 
This is the motivation for introducing a different approach.

\subsection{Dynamic setting: continuous time market}
First of all, we observe that, due to the Cauchy-Schwarz inequality, Definition \ref{no-good-deal measure} is equivalent to:
\begin{definition}
\label{no-good-deal measure 2}
A probability measure $Q$ equivalent to $P$ is a {\it no good-deal pricing measure} if 
$\frac{dQ}{dP} \in L_2({\cal F}_T)$ satisfies
\begin{equation}
E\Big[\big(\frac{dQ}{dP}-1\big)^2\Big] \leq \delta ^2.
\label{SR2}
\end{equation}
\end{definition}

\begin{lemma} 
\label{lm-measures}
Define the set ${\cal Q}_{st}$ of probability measures on $\mcal{F}_t$ as
$$
{\cal Q}_{st}:= \big\{Q \ll P \,| Q_{|{\cal F}_s}\hspace{-1mm}=\hspace{-1mm} P\: and \; \exists\, g_{st} \in L^2({\cal F}_t): \: \frac{dQ}{dP}=1+g_{st},\; E\big[g_{st}^2|{\cal F}_s\big] \leq \delta_{st}^2\big\}.
$$ 
Assume that the family of non negative real numbers $\delta_{st}$, $s,t\in [0,T]$: $s\leq t$, satisfies the following condition:
\begin{equation}
(\delta_{rs}\delta_{st}+\delta_{rs}+\delta_{st})= \delta_{rt}, \quad \forall r \leq s \leq t.
\label{eqComp}
\end{equation}
Then, 
\begin{enumerate}
\item for every probability measures $Q \in {\cal Q}_{rs}$ and $R \in {\cal Q}_{st}$, the probability measure $S\ll P$ with $\frac{dS}{dP} = \frac{dQ}{dP}\frac{dR}{dP}$ belongs to ${\cal Q}_{rt}$.\\
\item  for all $A \in {\cal F}_s$ and all $Q_1, Q_2 \in {\cal Q}_{st}$, there exists $Q \in {\cal Q}_{st}$ such that $ \frac{dQ}{dP}=\frac{d{Q_1}}{dP}1_A+\frac{d{Q_2}}{dP}1_{A^c}$.
\end{enumerate}
\label{lemmaSR3}
\end{lemma}

\begin{proof}
Note that
$$
\frac{dS}{dP}=(1+g_{rs})(1+g_{st})=1+g_{rs}g_{st}+g_{rs}+g_{st}
$$
and also that $E((g_{rs}g_{st}+g_{rs}+g_{st})^2|{\cal F}_r)$ $=$ $E(g_{rs}^2 E(g_{st}^2|{\cal F}_s)|{\cal F}_r)+E(g_{rs}^2|{\cal F}_r)+E(g_{st}^2|{\cal F}_r)+2E(g_{rs}^2g_{st}|{\cal F}_r)+2E(g_{st}^2g_{rs}|{\cal F}_r)+2E(g_{rs}g_{st}|{\cal F}_r)$.
Hence, using the properties of the conditional expectation, the inequalities of the kind
$ E(g_{st}^2|{\cal F}_s) \leq \delta_{st}^2$, and the Cauchy Schwarz inequality, we obtain that
$$
E\Big[\big(\frac{dS}{dP}-1\big)^2|{\cal F}_r\Big] 
\leq \big(\delta_{rs}\delta_{st}+\delta_{rs}+\delta_{st}\big)^2 
\leq \delta_{rt}^2.
$$
Thus $S \in {\cal Q}_{rt}$. 
The last assertion is obvious.
\end{proof}

\begin{remark}
An example of constants $\delta_{st}$, $s\leq t$, satisfying \eqref{eqComp} is given by
$\delta_{st}:=\delta^{t-s}-1$ for some $\delta >1$.
In fact, it is easy to see that the relation:
$$(1+\delta_{rt})=(1+\delta_{rs})(1+\delta_{st})$$
is satisfied.
Note also that, if $\delta$ is the Sharpe ratio bound as in \eqref{SR2}, then $\delta_{0T} = \delta^T - 1 = \delta$.
\end{remark}

In view of the previous result we can give the following definition of a dynamic Sharpe ratio.

\begin{definition}
\label{dSR}
A probability measure $Q$ equivalent to $P$ is a {\it dynamic no-good-deal pricing measure} if 
$\frac{dQ}{dP} \in L_2({\cal F}_T)$ satisfies
\begin{equation}
E\Big[\big(\Big(\frac{dQ}{dP}\Big)_t\Big(\frac{dQ}{dP}\Big)_s^{-1}-1\big)^2\vert \mcal{F}_s\Big] \leq \delta_{st}^2,
\end{equation}
for every $s\leq t$ and constants $\delta_{st}>0$ satisfying \eqref{eqComp}.
we recall that $\big(\frac{dQ}{dP}\big)_t:= E\big[ \frac{dQ}{dP} \vert \mcal{F}_t \big]$.
\end{definition}
It is immediate to see that, if $s=0,t=T$, then the definition above corresponds to the one in the static setting with $\delta_{0,T}=\delta$.

\vspace{2mm}
The next result gives a characterization of operators acting as majorant and minorant of prices.
These are directly connected to the dynamic Sharpe ratio bounds as in Definition \ref{dSR}.

\begin{proposition}
Let 
\[
\begin{split}
m_{st}(X) &:= \essinf_{Q \in {\cal Q}_{st}}E_Q[X|{\cal F}_s], \quad X\in L^+_\infty(\mcal{F}_t),\\
M_{st}(X) &:= \esssup_{Q \in {\cal Q}_{st}} E_Q[X|{\cal F}_s], \quad X\in L^+_\infty(\mcal{F}_t),
\end{split}
\]
where ${\cal Q}_{st}$ is defined as in Lemma \ref{lemmaSR3}.
Assume that the constants $\delta_{st}$, $s,t\in [0,T]: s\leq t$, satisfy the relation \eqref{eqComp} and that $\delta_{st}\rightarrow 0$, $t\downarrow s$. \\
Then $M_{st}(X)$, $X\in L^+_\infty(\mcal{F}_t)$, $s,t\in [0,T]: s\leq t$, is a weakly time-consistent, regular family of sublinear, monotone, $\mcal{F}_s$-homogeneous operators. Moreover \eqref{wrc} holds:
\[
M_{st}(X) \geq \liminf_{n\to \infty} M_{s_nt}(X), \quad X\in L^+_\infty(\mcal{F}_t).
\]
Furthermore $m_{st}(X)$, $X\in L^+_\infty(\mcal{F}_t)$, $s,t\in [0,T]: s\leq t$, is a weakly time-consistent, regular family of superlinear, monotone, $\mcal{F}_s$- homogeneous. Moreover, \eqref{wrc} holds:
\[
m_{st}(X) \leq \limsup_{n\to \infty} m_{s_nt}(X), \quad X\in L^+_\infty(\mcal{F}_t).
\]
\label{propSR3}
\end{proposition}
\begin{proof} 
For any $s,t\in [0,T]: s\leq t$, the properties of the operators $M_{st}(X)$, $m_{st}(X)$, $X\in L^+_\infty(\mcal{F}_t)$, are immediate.
For what concerns weak time-consistency, the proof of \eqref{wtcM} and \eqref{wtcm} is a simple adaptation of the proof of Theorem 4.4 of \cite{BN2008}:
$\{E_R(X|{\cal F}_s),\;R \in {\cal Q}_{st}\}$ is a lattice upward directed, from Lemma \ref{lemmaSR3} point {\it 2}. Then from Proposition VI.1.1 of \cite{Nev}, it follows that $\forall Q \in {\cal Q}_{st}$, 
$$E_Q(M_{st}(X)| {\cal F}_r) \leq \esssup_{R \in {\cal Q}_{st}}E_Q(E_R(X|{\cal F}_s)|{\cal F}_r))$$
From Lemma \ref{lemmaSR3} point {\it 1}, it follows that $M_{rs}(M_{st}(X)) \leq M_{rt}(X)$.\\
Hereafter we discuss the proof of \eqref{eqwtcM2}. The arguments can be easily adapted for the proof of \eqref{wtcm2}.\\
Let $t_n \downarrow t$ and 
consider  $Q\in \mcal{Q}_{st_n}$ with $\frac{dQ}{dP}= 1+g_{st_n}$ and consider a measure $Q'<<P$ on $\mcal{F}_t$ given by:
\[
\frac{dQ'}{dP}= 1+ k_{st}:= 1 + \frac{\delta_{st}}{\delta_{st_n}} E\big[g_{st_n}\vert \mcal{F}_t \big].
\]
Then we can see that
\[
E\Big[ \frac{dQ'}{dP} \vert \mcal{F}_s \Big] = 1 + \frac{\delta_{st}}{\delta_{st_n}} E\big[ g_{st_n}\vert \mcal{F}_s\big] = 1.
\]
Hence $Q'_{\vert \mcal{F}_s}=P$.
Moreover, we have 
\[
E\big[k^2_{st}\vert \mcal{F}_s\big] 
\leq \frac{\delta^2_{st}}{\delta^2_{st_n}}  E\big[g^2_{st_n}\vert \mcal{F}_s\big] 
\leq \delta^2_{st}. 
\]
Then we conclude that $Q'\in \mcal{Q}_{st}$.
Now consider $X\in L^+_\infty(\mcal{F}_t)$. We can see that
\[\begin{split}
E_{Q}\big[ X\vert \mcal{F}_s\big] &= E \big[ (1+g_{st_n}) X\vert \mcal{F}_s\big]\\
&= \frac{\delta_{st_n}}{\delta_{st}} E_{Q'}\big[ X\vert \mcal{F}_s\big]
+ \Big(1- \frac{\delta_{st_n}}{\delta_{st}} \Big) E \big[ X\vert \mcal{F}_s\big].
\end{split}
\]
Thus we obtain
\begin{equation}
M_{st_n}(X) \leq  \frac{\delta_{st_n}}{\delta_{st}}  M_{st}(X) + \Big(1- \frac{\delta_{st_n}}{\delta_{st}} \Big) E \big[ X\vert \mcal{F}_s\big].
\label{eqMdel}
\end{equation}
Taking the limit for $n\to\infty$, we can see that $\delta_{st_n}\to \delta_{st}$, as a direct application of \eqref{eqComp} and the assumptions. 
On the other side we have that $M_{st}(X) \leq M_{st_n}(X)$ as $\mcal{Q}_{st}\subseteq \mcal{Q}_{st_n}$. 
Thus it follows  from (\ref{eqMdel}) that $\forall X \in L^+_{\infty}({\cal F}_t),\;\;M_{st}(X)= \lim_{n\to \infty} M_{st_n}(X)$.\\
At last we prove relationships \eqref{wrc}. As before we study the majorant operators only as the arguments for the minorant operators are easily adapted. Let $s<s_n<t$, $s_n \downarrow s$. 
From weak time-consistency we know that for all $X \in L^+_{\infty}({\cal F}_t)$, $M_{st}(X) \geq M_{ss_n}(M_{s_nt}(X))$.
Moreover, for every $Y\geq 0$, we have $M_{ss_n}(Y) \geq E\big[ Y\vert\mcal{F}_s\big]$ as $P\in \mcal{Q}_{ss_n}$.
Thus we have
\[
M_{st}(X) \geq E\big[ M_{s_nt}(X)\vert\mcal{F}_s\big] 
\geq E\big[ \liminf_{n\to \infty} M_{s_nt}(X)\vert\mcal{F}_s\big] 
\]
and $\liminf_{n\to \infty} M_{s_nt}(X)$ is $\mcal{F}_s$-measurable as the filtration is right-continuous.
By this we end the proof.
\end{proof}

Finaly we get the following result.
\begin{theorem}
Let $x_{st}$, ${0 \leq s \leq t \leq T}$, be  a right-continuous time-consistent family of price operators defined on the linear space of marketed financial assets $L_t$. Assume that the family $x_{st}$ satisfies the following sandwich condition:
$$m_{s,t}(Z)+x_{s, t}(X) \leq M_{s, t}(Y)$$
for all $X \in L_{t}$ and for every $Y,Z, \in L^+_{\infty}({\cal F}_{t})$: $Z+X \leq Y$, where $M_{st}$ and $m_{st}$ are defined as in Proposition \ref{propSR3}.
Assume  that $m_{0T}(X)>0$ for every $X>  0$.
Then  $x_{st}$ extends to a time-consistent family of linear price operators $\hat x_{st}$ defined on all $L_{\infty}({\cal F}_t)$ with values in $L_{\infty}({\cal F}_s)$, such that
\begin{equation}
\label{sand-ngd}
m_{st}(X) \leq \hat x_{st}(X) \leq M_{st}(X), \quad X \in L^+_{\infty}({\cal F}_T),
\end{equation}
and admitting representation
$$
\hat x_{st}(X)=E_{P_0}(X|{\cal F}_s),\;\; \quad X \in L_{\infty}({\cal F}_t),
$$
for some dynamic no-good-deal pricing measure $P_0$ equivalent with $P$.
\label{thmSR3}
\end{theorem}

\begin{proof}
By Proposition \ref{propSR3}, we can apply Theorem \ref{thmP} and obtain 
$\hat x_{st}(X)=E_{P_0}(X|{\cal F}_s)$, $X \in L_{\infty}({\cal F}_t)$ and \eqref{sand-ngd} for $X \geq 0$.
Next we show that $P_0$ is satisfying Definition \ref{dSR}.
Consider $X \in L_{\infty}({\cal F}_t)$.
Then, from \eqref{sand-ngd}, we see that
\[
E_{P_0}\big[ X \vert \mcal{F}_s \big] \leq \esssup_{Q\in \mcal{Q}_{st}}  E_{Q}\big[ X \vert \mcal{F}_s \big].
\]
Recall that, for any $Q\in \mcal{Q}_{st}$, we have $\frac{dQ}{dP} = 1+ g_{st}$ and $E[g^2_{st}\vert \mcal{F}_s] \leq \delta^2_{st}$.
Similarly we denote the conditional Radon-Nikodym derivative of $P_0$ by 
$\Big( \frac{dP_0}{dP} \Big)_t  \Big( \frac{dP_0}{dP} \Big)^{-1}_s = 1 + k_{st}  $.
Hence we have
\[
\begin{split}
E\big[ k_{st} X \vert \mcal{F}_s\big] 
&\leq \esssup_{Q\in \mcal{Q}_{st}}  E\big[g_{st} X \vert \mcal{F}_s \big]\\
&\leq \esssup_{Q\in \mcal{Q}_{st}}  \sqrt{ E\big[g_{st}^2 \vert \mcal{F}_s \big]\, E\big[X^2 \vert \mcal{F}_s \big]  }\\
& \leq \delta_{st} \sqrt{  E\big[X^2 \vert \mcal{F}_s \big]  }.
\end{split}
\]
We can then conclude that
\[
\sqrt{  E\big[k_{st}^2 \vert \mcal{F}_s \big]  }
=\sup_{X\in L_\infty(\mcal{F}_t)} \frac{E\big[ k_{st} X \vert \mcal{F}_s\big] }{\sqrt{  E\big[X^2 \vert \mcal{F}_s \big]  }} \leq \delta_{st}.
\]
By this the proof is complete.
\end{proof}

We can therefore give the following definition.
\begin{definition}
A dynamic no-good-deal price system $x_{st}(X)$, $X\in L_\infty(\mcal{F}_t)$, $0\leq s \leq t\leq T$, is a time-consistent family of linear price operators satisfying the dynamic sandwich condition:
$$m_{st}(X) \leq x_{st}(X) \leq M_{st}(X)$$ for every $X \in L^+_{\infty}({\cal F}_T)$,
where $m_{st}$ and $M_{st}$ are defined as in Proposition \ref{propSR3}.
\end{definition}

%%%%%%%%%%%%%%%%%%%%%%%%%%
\section{Appendix: Filters on a topological space and compactness}

In this section we briefly report the basic definitions related to filters in general topological spaces as in \cite{B90}. These notions are used in the proof of Theorem \ref{thmP}.

\begin{defn} \label{filter}
A \emph{filter} $\U$ on the set $D$ is a non empty family of subsets of $D$ satisfying the following properties:
\begin{itemize}
\item
Any $C\subseteq D$ for which there exists $U\in \U$ such that $C \supseteq U$ belongs to $\U$;
\item
For any $U_k \in \U$, $k=1,...,K$ ($K\in \mathbb{N}$), the set $\bigcap_{k=1}^K U_k$ belongs to $\U$;
\item
$\emptyset \notin \U$.
\end{itemize}
\end{defn}

Note that from the two last properties we see that any finite intersection of elements of the filter is non-empty.

{\bf Example.} If $D$ is a topological space, then the family of all neighbourhoods $V(f)$ of a point $f\in D$ is a filter.

\begin{defn}
A non empty family $\B$ of subsets of $D$ is a \emph{filter base on $D$}  if the intersection of a finite number of elements of $\B$ contains an element of $\B$ and  $\emptyset \notin \B$.
\end{defn}

Note that any filter is a  filter base . Moreover, form the definition, it is easy to see that if $\U$ is a filter base on $D$ and
$E\supseteq D$, then $\U$ is a  filter base on $E$.

\vspace{2mm}
Let $D$ be a topological space and consider the subset $C\subseteq D$. Recall that a point $f\in D$ is called adherent to $C$ if for every neighbourhood $V(f)$ of $f$ it is verified that $V(f) \cap C \ne \emptyset$.

\begin{defn}
Let $D$ be a topological space. A point $f\in D$ is \emph{adherent to the  filter base $\B$ on $D$} if for every neighbourhood $V(f)$ and every $U \in \B$ it is verified that $V(f) \cap U \ne \emptyset$.
\end{defn}

\begin{defn}
Let $D$ be a topological space satisfying the Hausdorff separation axiom. We say that $D$ is \emph{compact} if for any filter on $D$ there exists an adherent point.
\end{defn}

%%%%%%%%%%%%%%%%%%
\vspace{3mm}
{\bf Acknoledgments.}
This research was specially carried through during the visit of G. Di Nunno at Ecole Polytechnique with the support of Chair of Financial Risks of the Risk Foundation, Paris, and the visit of J. Bion-Nadal at University of Oslo with the support of CMA - Centre of Mathematics for Applications.

%%%%%%%%%%%
\vspace{2mm}
\bibliographystyle{plain}

\end{document}